\documentclass[smallextended]{svjour3}
\smartqed 
\usepackage{graphicx}
\usepackage{amsmath}
\usepackage{amsfonts}
\usepackage{amssymb}
\usepackage{enumerate}
\usepackage{slashbox}
\usepackage{subfigure}
\usepackage{textcomp,enumerate}
\begin{document}
\title{A mathematical programming based characterization of Nash equilibria of 
some constrained stochastic games\footnote[1]{A portion of Section \ref{IND_model} (two player case) has been presented in the 8th International ISDG workshop at University of Padova, Italy on 21-23 July, 2011.}}
 
\titlerunning{Constrained stochastic games}

\author
{Vikas Vikram Singh         
\and N. Hemachandra
}

\institute
{Vikas Vikram Singh
\at
Industrial Engineering and Operations Research, Indian Institute of Technology Bombay,
Mumbai 400076, India \\
\email{vikas$_{-}$singh@iitb.ac.in}           
\and
N. Hemachandra 
\at
Industrial Engineering and Operations Research, Indian Institute of Technology Bombay,
Mumbai 400076, India \\
\email{nh@iitb.ac.in}  
}

\date{Received: date / Accepted: date}

\maketitle
\begin{abstract}
We consider two classes of constrained finite state-action stochastic
games. First, we consider a two player nonzero sum single controller
constrained stochastic game with both average and discounted
cost criterion. We consider the same type of constraints as in \cite{Altman},
i.e., player 1 has subscription based constraints and player 2, who
controls the transition probabilities, has realization based constraints
which can also depend on the strategies of player 1. Next, we consider a
$N$-player nonzero sum constrained stochastic game with independent
state processes where each player has average cost criterion as
discussed in \cite{Altman3}. We show that the stationary Nash equilibria of both
classes of constrained games, which exists under strong Slater and
irreducibility conditions \cite{Shwartz}, \cite{Altman3}, has one to one correspondence with
global minima of certain mathematical programs. In the single controller
game if the constraints of player 2 do not depend on the strategies of
the player 1, then the mathematical program reduces to the non-convex
quadratic program. In two player independent state processes stochastic
game if the constraints of a player do not depend on the strategies of
 another player, then the mathematical program reduces to a
non-convex quadratic program. Computational algorithms for finding
global minima of non-convex quadratic program exist \cite{Giannessi}, \cite{Jing} and hence,
one can compute Nash equilibria of these constrained stochastic games.
Our results generalize some existing results for zero sum games \cite{Altman}, \cite{Hordijk2}, \cite{Miller}.
\end{abstract}

\section{Introduction}
It is well known that there is a substantial relationship between game theory and mathematical programming. While it is well  known that equilibrium strategies in two player zero sum matrix games are related to optimal points of certain linear programs, in 1964, Mangasarian and Stone \cite{Stone} have shown that the Nash equilibria of any two player bimatrix game can be obtained from the global maxima of one quadratic program and this approach can be generalized in case of any finite number of players. Later Filar et al. \cite{Thuijsman}, generalized this idea to the infinite horizon stochastic game with finite state space and finite action spaces of all the players.
It has been shown that the stationary Nash equilibria of any $N$-player stochastic game with discounted criterion are in one to one correspondence with the global minima of a certain mathematical program  \cite{Thuijsman}, \cite{Filar}; so, Nash equilibria of such a stochastic game can be computed via the global minima of one mathematical program. 
The stochastic games described in \cite{Thuijsman}, \cite{Filar} can be viewed as centralized stochastic games. In such centralized stochastic games all the players jointly control a single Markov chain and all the players have complete information of the Markov chain's state and for taking decision at any time $t$ each player has
information of all the actions previously taken by the players. The review article \cite{Raghavan} summarizes various algorithmic aspects of zero sum stochastic games along with algorithms for nonzero sum stochastic games with special structure (single controller, etc.). In particular, two player zero sum single controller stochastic game can be solved by a linear program 
\cite{Parth}, \cite{Filar}, \cite{Vrieze2} and the Nash equilibria of the nonzero sum single controller stochastic game can be obtained from the global minima of a quadratic program \cite{Filar_2}.

Since the seminal work of Lloyd S. Shapley \cite{Shapley}, stochastic games have come to constitute
an important class of models that can capture game theoretic issues among the decision
makers involved, apart from accounting for random evolution of the system. The edited volume by Neyman and Sorin \cite{Neyman} has a nice collection of many articles on stochastic games and their applications. The book by Filar and Vrieze \cite{Filar} presents stochastic games as a natural multi-player generalization of (single player) Markov decision processes and their applications. 
Constrained stochastic games are realistic because they can capture bounds on consumption of resources, but, are also difficult to analyze. In \cite{Shwartz} $N$-player centralized constrained stochastic games with both discounted and average cost criterion
with finite state and finite action spaces are considered and it is shown that there exists a stationary
Nash equilibrium under strong Slater condition (irreducibility assumption is also needed in average case). 
The existence of Nash equilibrium for constrained stochastic games 
when the state space is countable and action spaces are compact metric space is discussed in \cite{Hernandez}.
The characterization of Nash equilibria for general constrained stochastic games  via some mathematical program is not known. 
To the best of our knowledge there are only some special classes  of constrained stochastic games which can be solved as linear programs.  We give a brief description of all these classes here.  
The two player zero sum single controller constrained stochastic game with total expected reward criterion
and expected average reward criterion is considered in \cite{Hordijk1}, \cite{Hordijk2} respectively. In both 
\cite{Hordijk1}, \cite{Hordijk2} only the player who controls the transition probabilities has 
constraints on his expected rewards and these rewards do not depend on the strategies of the other player. Nash equilibrium of such stochastic games can be computed from optimal solutions of linear programs. 
Altman, et al., \cite{Altman} considered the zero sum constrained stochastic game with 
discounted cost criterion where both the players have constraints. The player who controls the 
transition probabilities has constraints on his expected discounted costs as similar in 
\cite{Hordijk1}, \cite{Hordijk2} and other player has subscription based constraints. This class of games also can be solved by linear programs \cite{Altman}.
 
Apart from the centralized stochastic games as discussed above, some 
decentralized stochastic games are being considered in the literature recently \cite{Miller}, \cite{Altman3}, \cite{Altman4}. 
In decentralized stochastic games each player independently controls his own Markov chain based on his state and actions. 
In \cite{Altman3}, a $N$-player decentralized constrained stochastic game with average cost criterion 
is considered  and it is shown that the Nash equilibrium for these games exists in 
stationary strategies under the irreducibility and strong Slater condition. 
In these games each player controls his own  Markov chain and the constraints of each player 
 depend also on the strategies of all the players. The application of these games to modeling of wireless network is described in \cite{Miller}, \cite{Altman3}, \cite{Altman4}. 
Two player zero sum game of this class where the constraints of each player do not depend on the other 
player's strategies is considered in \cite{Miller}. These games, with both unichain and multichain structure on the state processes of both the players, can be solved by linear programs. 
 
In this paper we consider two different classes of constrained stochastic games. First, we consider a special class of two player nonzero sum centralized constrained stochastic games which is a single controller constrained stochastic game with both average and discounted cost criterion. We then consider a $N$-player 
nonzero sum constrained stochastic game with independent state processes where all the players use average cost criterion as discussed in \cite{Altman3}. The summary of our results are:
\begin{enumerate}
\item We consider a two player nonzero sum single controller constrained stochastic game with both average and discounted   cost criterion,  a special class of centralized constrained stochastic games, with the 
same type of constraints as in \cite{Altman}, i.e., player 1 has subscription
based constraints and player 2, who controls the transition probabilities, has realization based constraints. Unlike the situation in \cite{Altman} and \cite{Hordijk2}
we consider the case where realization based constraints of player 2 depend on the strategies of both the 
players. It follows from \cite{Shwartz} that there exists a stationary Nash equilibrium under  
strong Slater condition (irreducibility assumption is also needed in average case). 
We show that the Nash equilibria of this constrained stochastic game can be obtained from the 
global minima of one mathematical program. 
The converse statement is also true, i.e., from the stationary Nash equilibrium 
of these games we can construct a point which 
is a global minimum of the corresponding mathematical program.
\item  If the constraints of player 2 do not depend on the strategies of player 1, then the mathematical program 
reduces to the non-convex quadratic program. For zero sum case the linear programs given in \cite{Altman}, \cite{Hordijk2} can be recovered from our quadratic program. 
\item We show that the stationary Nash equilibria of $N$-player nonzero sum constrained stochastic game with independent state processes \cite{Altman3} can be obtained from the global minima of a certain mathematical program. 
The converse statement is also true, i.e., the stationary Nash equilibrium of these games,  which exists under strong Slater and irreducibility conditions \cite{Altman3}, corresponds to a point which is a global minimum of the  corresponding mathematical program.  
\item In two player constrained stochastic game with independent state processes case, if the constraints of each player  do not depend on the other player's strategies, then the corresponding  mathematical program reduces to the non-convex quadratic program. The linear program as given in \cite{Miller} for zero sum game can be obtained as a special case of our quadratic program.
\end{enumerate}
To derive mathematical programs for both constrained stochastic games we use the same approach, which is via best response linear programs. We use the fact that the best response of each player against the fixed strategies of other players can be obtained by solving a constrained Markov decision model, which, in turn, can be obtained by solving a linear program \cite{Altman2}. In both the cases due to some special structure we are able to put all primal-dual pair of linear programs (one pair for each player) together to form one mathematical program whose objective function is nonnegative at all feasible points. As the linear program which gives the optimal strategy in a constrained Markov decision model is given in terms of occupation measure, our mathematical programs are in terms of these occupation measures. The Nash equilibrium strategy can be recovered from occupation measure by a known transformation~\cite{Altman2}.   

There are some methods available for solving non-convex quadratic programming problem 
\cite{Giannessi}, \cite{Cottle}, \cite{Burdet}. The algorithm given in \cite{Burdet} 
is based on complete enumeration of the faces of the polyhedron and therefore it is
not very efficient while the cutting plane method of \cite{Cottle} seems to be problematic \cite{Zwart}.  
The algorithm given in \cite{Giannessi} to solve quadratic programs terminates in a finite number of steps. 
We note that the  algorithm of \cite{Giannessi} assumes that quadratic program has a global minimum 
and this condition is satisfied in our settings.  In \cite{Jing}, one more algorithm based on linear programming with complementarity constraints approach is given to solve a  non-convex quadratic program. This algorithm does not assume the quadratic program to be bounded below on feasible set.
(If quadratic program is not bounded below, then the algorithm given in \cite{Jing} computes a feasible ray on which the quadratic program is unbounded;  otherwise, it finds an optimal solution in finite number of steps). But, in our case the quadratic programs are bounded below on feasible set and hence the algorithm given in \cite{Giannessi} is applicable to our settings. One can also attempt to use general purpose nonlinear solvers to solve these non-convex  quadratic programs, but convergence to global minima may not be guaranteed.

We now describe the structure of the rest of our paper. Section \ref{SGCT} contains the two player nonzero sum single controller constrained stochastic game with both average and discounted cost criteria and its mathematical programming formulation. 
Section \ref{IND_model} contains $N$-player constrained stochastic game with independent state processes with average cost criterion and its mathematical programming formulation.

\section{Single controller constrained stochastic game}\label{SGCT}
We consider two player nonzero sum single controller constrained stochastic games
with both  average  and discounted cost criterion. We assume that player 2 controls the Markov chain. As similar to \cite{Altman},  player 1 has subscription based constraints and player 2 has realization based constraints but unlike the case in 
\cite{Altman}, \cite{Hordijk2} the constraints of player 2 can also depend on the strategies of player 1. 
This class of stochastic game is described by the following objects:   
\begin{enumerate}
\item[(i)]  $S$ is finite state space of the game. 
The generic element of $S$ is denoted by~$s$.

\item[(ii)] $\gamma=\left(\gamma(1),\gamma(2),\cdots,\gamma\left(|S|\right)\right)$ is a probability distribution over $S$
according to which initial state 
is chosen.

\item[(iii)] $A^{i}$ is the finite action set of player $i$, $i=1,2$, let $A^i(s)$ denotes the set of 
actions available to player $i$ when the state is at $s$, where
$A^i=\bigcup_{s\in S} A^i(s)$.
\item[(iv)] Define, $\mathcal{K}=\left\{(s,a^1,a^2):s\in S,a^1\in A^1(s),a^2\in A^2(s)\right\}$ and for $i=1,2$, $\mathcal{K}^i=\left\{(s,a^i):s\in S, a^i\in A^i(s)\right\}$.
\item[(v)]  $c^i:\mathcal{K} \rightarrow \mathbb{R}$ is immediate cost of player $i$, $i=1,2$. 
  Specifically, $c^i(s,a^1, a^2)$ is the immediate cost incurred by player $i$, $i=1,2$,  when state is $s\in S$ 
and actions chosen by player 1 and player 2 are $a^1 \in A^1(s)$ and  $a^2\in A^2(s)$ respectively. Player $i$
wants to minimize the expected cost involving $c^i(\cdot)$, $i=1,2$. 

\item[(vi)] $d_{sub}^{1,k}:\mathcal{K}^1\rightarrow \mathbb{R}$ is subscription type cost of  player 1.
$d^{1,k}_{sub}(s,a^1)$ denotes subscription cost 
which player 1 has to pay for using action 
$a^1$ at state $s$ for $k$th service, $k=1,2,\cdots,n_1$.

\item[(vii)] $d^{2,l}:\mathcal{K} \rightarrow \mathbb{R}$ is immediate cost of player 2.
These $d^{2,l}(\cdot)$ are involved in the $l$th, $l=1,2,\cdots,n_2$, constraint on expected cost of player 2. 

\item[(viii)] Define, $\wp(M)$ as set of all probability measures over set $M$. $p:\mathcal{K}^2\rightarrow~\wp(S)$ is transition probability describing the dynamics of the game, where  
 $p(s'|s,a^2)$ is a probability of going to state $s'$ from state $s$ when player 2 
chooses action  $a^2\in A^2(s)$. 
We recall that the game is controlled by only player 2.  

\item[(ix)] $\xi^1=\left(\xi_1^1,\xi_2^1,\cdots,\xi_{n_1}^1\right)^T$, $\xi^2=\left(\xi_1^2,\xi_2^2,\cdots,\xi_{n_2}^2\right)^T$ 
denote the vectors defining the given bounds of the constraints on both the 
players.

\end{enumerate} 

The game dynamics are as follows.  Initially, at time $t=0$, the state of the game is $s$ which is chosen according to initial distribution $\gamma$,
player 1 chooses an action $a^{1}\in A^{1}(s)$ and player 2 chooses an action
$a^{2}\in A^{2}(s)$ independent of each other. Player 1 receives
an immediate cost of $c^{1}(s,a^{1},a^{2})$ and player 2 receives $c^{2}(s,a^{1},a^{2})$. Apart from this player 2 receives another immediate costs $\{d^{2,l}(s,a^1,a^2)\}$, $l=1,2,\cdots,n_2$, 
which are involved in the expected cost functionals of player 2  that are constrained by specified bounds $\{\xi_l^2\}$, $l=1,2,\cdots,n_2$.
Now, the state of the game switches to a new state $\hat{s}\in S$  at time $t=1$ with  probability $p(\hat{s}|s,a^2)$.  
At time $t=1$, in state $\hat{s}$, player 1 chooses an action $\hat{a}^1\in A^1(\hat{s})$ and player 2 chooses an action $\hat{a}^2\in A^2(\hat{s})$ and receives immediate  cost $c^1(\hat{s},\hat{a}^1,\hat{a}^2)$ and $c^2(\hat{s},\hat{a}^1,\hat{a}^2)$ respectively, player 2 also receives immediate costs $\{d^{2,l}(\hat{s},\hat{a}^1,\hat{a}^2)\}$, $l=1,2,\cdots,n_2$. The next state of the game is $\tilde{s}\in S$ with probability 
$p(\tilde{s}|\hat{s},\hat{a}^2)$.
The same thing repeats at state $\tilde{s}$ and play continues for infinite time horizon.

While transition probabilities depend only on the present state and action used, action that are used can depend on `past', as in history dependent strategies.
Define a history at time $t$ as 
$h_t=(s_0,a_0^1,a_0^2,s_1,a_1^1,a_1^2,\cdots,s_{t-1},a_{t-1}^1,a_{t-1}^2,s_t)$
where $s_t\in S$, $a_t^i\in A^i(s_t)$, $i=1,2$, $t=0,1,2,\cdots$. Let $H_t$ denote the set of all possible histories of length $t$.
A decision rule $f_t:H_t\rightarrow \wp(A^1(s_t))$ (resp., $g_t:H_t\rightarrow \wp(A^2(s_t))$) 
of player 1 (resp., player 2) at time $t$ is a 
function that assigns to any history of length $t$, a probability measure over action set of player 1 (resp., player 2).
This means that under decision rule $f_t$ (resp., $g_t$), player 1 (resp., player 2) chooses action $a^1$ (resp., $a^2$) with probability $f_t(h_t,a^1)$ (resp., $g_t(h_t,a^2)$). 
The sequence of decision rules is called the strategy of the player.
$f^h= (f_0,f_1,\cdots,f_t,\cdots)$ 
and $g^h=(g_0,g_1,\cdots,g_t,\cdots)$ denote the strategies of player 1 and
player 2 respectively and are called
history dependent (behavioral) strategies.

 Let $F$ and $G$ denote the set of all history dependent strategies 
of player 1 and player 2 respectively.  
These strategies are called Markovian if at every decision epoch the decision rule depends 
only on the current state but the decision rule can differ at every epoch. 
A stationary strategy is a Markovian strategy which does not depend on the time, i.e.,
at every decision epoch the decision rule is same. So, for stationary strategy 
$f_t=f$ and $g_t=g$ for all $t$, i.e., $(f,f,f,\cdots)$ 
and $(g,g,g,\cdots)$ are the stationary strategies of player 1 
and player 2 respectively. We denote, with 
some abuse of notations, $f$ and 
$g$ as stationary strategies of player 1 and player 2 respectively. 
Let $F_S$ and $G_S$ denote the set of all stationary strategies of player 1 and player 2 respectively.
A stationary strategy $f\in F_S$ is identified with $f=\left((f(1))^T,(f(2))^T,\cdots,(f\left(|S|\right))^T\right)^T$, 
where $f(s)=\left(f(s,1),f(s,2)\cdots,f\left(s,|A^1(s)|\right)\right)^T$ for all $s\in S$; $|M|$ denotes the cardinality of set $M$.
Similarly, $g$ is identified with  $g=\big((g(1))^T,(g(2))^T,\cdots,$ $(g\left(|S|\right))^T\big)^T$, where 
$g(s)=\left(g(s,1),g(s,2)\cdots,g\left(s,|A^2(s)|\right)\right)^T$ for all $s\in S$. For all $s\in S$, $f(s,a^1)$ is then the probability of choosing action $a^1\in A^1(s)$ by player 1 and $g(s,a^2)$ is probability of choosing action $a^2\in A^2(s)$ by player 2 when state is $s$.

This leads to
the introduction of vector stochastic process $\{X_t,\mathbb{A}_t^1,\mathbb{A}_t^2\}_{t=0}^\infty$, where,
$X_t$ denotes the state of the game, $\mathbb{A}_t^1$, the action chosen by player 1 
and $\mathbb{A}_t^2$,
the action chosen by the player 2 at  time $t$.
An initial distribution $\gamma$ together with strategy pair $(f^h,g^h)\in F\times G$ defines a unique probability measure 
$\mathbb{P}_{f^h,g^h}^\gamma$ on an appropriate probability space with respect to which the laws of 
vector stochastic process 
$\{X_t,\mathbb{A}_t^1,\mathbb{A}_t^2\}_{t=0}^\infty$ of states and actions can be defined. The corresponding expectation operator on this probability space is denoted by $\mathbb{E}_{f^h,g^h}^\gamma$.                

\subsubsection*{The expected average cost}
These costs are average functionals of states and actions of the game and each player minimizes his cost functionals.
For given initial distribution $\gamma$ and strategy pair
 $(f^h,g^h)$ the expected average cost of player $i$, $i=1,2$, is defined as
\begin{equation}\label{SGCT_avg_cost}
 C_{ea}^i(\gamma,f^h,g^h) = \limsup _{T\rightarrow\infty}\frac{1}{T}\sum_{t=0}^{T-1}
\mathbb{E}_{f^h,g^h}^\gamma c^i(X_t,\mathbb{A}_t^1,\mathbb{A}_t^2) 
\end{equation}
where $ea$ stands for expected average.
\subsubsection*{The expected average constraints}
The expected average constraints of player 2 are defined by average functionals of
states and actions of the game which are bounded by given reals.
For given initial distribution $\gamma$ and strategy pair
 $(f^h,g^h)$ the expected average costs of player 2 are defined as
\begin{equation*}\label{SGCT_avg_cost2}
 D_{ea}^{2,l}(\gamma,f^h,g^h)= \limsup _{T\rightarrow\infty}\frac{1}{T}\sum_{t=0}^{T-1}\mathbb{E}_{f^h,g^h}^\gamma d^{2,l}(X_t,\mathbb{A}_t^1,\mathbb{A}_t^2),\;\; \forall\; l=1,2,\cdots,n_2.
\end{equation*}
 $D_{ea}^{2,l}(\cdot,\cdot)$ can capture the average consumption of resource 
$l$, $l=1,2,\cdots,n_2$, by player 2. The expected average constraints of player 2 are given by
\begin{equation}\label{SGCT_avg_const2}
 D_{ea}^{2,l}(\gamma,f^h,g^h)\leq \xi_l^2, \;\; \forall\; l=1,2,\cdots,n_2.
\end{equation}
A constraint in \eqref{SGCT_avg_const2} captures the fact that the average consumption of resource $l$ by player 2, when
player 1 uses strategy $f^h$ and player 2 uses strategy $g^h$ is not more than given constant $\xi_l^2$, $l=1,2,\cdots,n_2$.

\subsubsection*{The expected discounted cost} 
These costs are discounted functionals of states and actions of the game and each player minimizes his cost functionals.
For given initial distribution $\gamma$ and strategy pair
 $(f^h,g^h)$ the expected discounted cost of player $i$, $i=1,2$, is defined as 
\begin{equation}\label{SGCT_cost}
C_\beta^i(\gamma,f^h,g^h)= (1-\beta)\sum_{t=0}^\infty \beta^t\mathbb{E}^\gamma_{f^h,g^h}
c^i(X_t,\mathbb{A}_t^1,\mathbb{A}_t^2)
\end{equation}
where $\beta\in [0,1)$ is a fixed discount factor.
\subsubsection*{The expected discounted constraints}
The expected discounted constraints of player 2 are defined by discounted functionals of
states and actions of the game which are bounded by given reals.
For given initial distribution $\gamma$ and strategy pair
 $(f^h,g^h)$ the expected discounted costs of player 2 are defined as
\begin{equation*}\label{SGCT_cost2}
 D_\beta^{2,l}(\gamma,f^h,g^h)= (1-\beta)\sum_{t=0}^\infty \beta^t\mathbb{E}^\gamma_{f^h,g^h}d^{2,l}(X_t,\mathbb{A}_t^1,\mathbb{A}_t^2),
\;\; \forall\; l=1,2,\cdots,n_2.
\end{equation*}
$D_\beta^{2,l}(\cdot,\cdot)$ can capture the discounted cost for the consumption of resource $l$, 
 $l=1,2,\cdots,n_2$, by player 2. The expected discounted constraints of player 2 are given by
\begin{equation}\label{SGCT_disc_const2}
 D_\beta^{2,l}(\gamma,f^h,g^h)\leq \xi_l^2, \;\; \forall\; l=1,2,\cdots,n_2.
\end{equation}
A constraint in \eqref{SGCT_disc_const2} captures the fact that discounted cost for the consumption of resource $l$ by player 2, when
player 1 uses strategy $f^h$ and player 2 uses strategy $g^h$ is not more than given real $\xi_l^2$, $l=1,2,\cdots,n_2$.

\subsubsection*{Subscription type cost \cite{Altman}}
The subscription type costs of player 1 are defined as in \cite{Altman} 
\begin{equation*}
 D_{sub}^{1,k}(f)=\sum_{s\in S}\sum_{a^1\in A^1(s)}d_{sub}^{1,k}(s,a^1)f(s,a^1)
\end{equation*}
for all $k=1,2,\cdots, n_1$ and $f\in F_S$. These costs are called as subscription type because they are 
 based only on the fraction of time during which a given action is used at a given state and are not based 
on how frequently the state is visited and action is used. 
This situation can arise where for using some services there is 
subscription/registration fee for their planned use and that can be paid in advance. 
\subsubsection*{Subscription type constraints}
The subscription type constraints of player 1 are defined as 
\begin{equation}\label{SGCTsub}
D_{sub}^{1,k}(f)\leq \xi_k^1, \; \;\forall \; k=1,2,\cdots,n_1.
\end{equation}

We denote $C^i$, $i=1,2$ and  $D^{2,l}$, $l=1,2,\cdots, n_2$,  
as expected costs which can be either average or discounted that depends on the criterion being used. 
Under average cost criterion both players have expected average costs and under discounted cost criterion both players have expected discounted costs. Apart from this, player 1 has
subscription based costs which are constrained by some given reals. 
It is clear that player 1 has $n_1$ number of constraints which are defined by \eqref{SGCTsub} and player 2 has $n_2$ number of constraints which are defined as  
\begin{equation}\label{SGCTreal}
D^{2,l}(\gamma,f^h,g^h)\leq \xi_l^2, \; \;\forall \; l=1,2,\cdots,n_2.
\end{equation}
The constraints \eqref{SGCTsub} and \eqref{SGCTreal} are called subscription based and realization
based constraints respectively. Both the players choose their actions independently and want to minimize their 
expected cost subject to their constraints from \eqref{SGCTsub} and \eqref{SGCTreal}. We denote this constrained
stochastic game by $G^c$.
As Nash equilibrium exists in stationary strategies 
under assumptions (A1)-(A2) given below \cite{Shwartz}, from now onwards we restrict ourselves to the stationary strategies.

The strategy pair $(f,g)$ is called $1$-feasible if it satisfies \eqref{SGCTsub} and strategy pair $(f,g)$ 
is called $2$-feasible if it satisfies \eqref{SGCTreal}. 
As the player 1 constraints \eqref{SGCTsub} do not depend on the strategies of player 2, 
then strategy pair $(f,g)$ is  $1$-feasible for all $g\in G_S$ if $f$ satisfies \eqref{SGCTsub}.
A strategy pair $(f,g)$ is called feasible if it is both $1$-feasible and $2$-feasible. 
Let $F^\xi_S$ denote the set of all feasible stationary  strategy pairs for the constrained stochastic game $G^c$. We shall assume throughout that $F^\xi_S$ is non-empty. 
Now, we recall the definition of Nash equilibrium as given in \cite{Shwartz}. 
A strategy pair $(f^*,g^*)\in F^\xi_S$ is called 
the Nash equilibrium of constrained 
stochastic game $G^c$ if it satisfies the following conditions
\begin{gather}\label{SGCT_NE_def1}
 C^1(\gamma,f^*,g^*)\leq C^1(\gamma,f,g^*), \; \; \forall \; \mbox{1-feasible}\; (f,g^*)\\ \label{SGCT_NE_def2}
C^2(\gamma,f^*,g^*)\leq C^2(\gamma,f^*,g^h), \; \; \forall \; \mbox{2-feasible} \; (f^*,g^h).
\end{gather}
Thus, unilateral deviation of any player $i$, $i=1,2$, will either violate the constraints of $i$th player,
or if it does not, it will result in a cost $C^i$ for that player that is not lower than the
one achieved by feasible strategy pair $(f^*,g^*)$.
The strategy pair $(f^*,g^*)\in F_S^\xi$ satisfying \eqref{SGCT_NE_def1} and \eqref{SGCT_NE_def2} would still be Nash equilibrium of constrained stochastic game if we replace strategy $g^h$ by stationary strategy $g$ in \eqref{SGCT_NE_def2}. This can be seen by noticing that when strategy of player 1 is fixed as a stationary strategy $f^*$, then player 2 is faced with a constrained Markov decision process (CMDP) where optimal strategy always exists in the space of stationary strategies \cite{Altman2}.

\subsubsection*{Assumptions [Altman and Shwartz \cite{Shwartz}]}
\begin{enumerate}
\item[(A1)] Ergodicity: In case of average cost criterion the unichain ergodic structure holds, i.e., under
every stationary strategy $g$ the state process is an irreducible Markov chain with one ergodic class (and possibly some transient states).
\item[(A2)] Strong Slater condition:  
For player 2, there exists some $g'$ such that for any strategy $f$ of player 1, 
\begin{equation*}\label{SGCT_Slater}
 D^{2,l}(\gamma,f,g')<\xi_l^2, \; \; \forall \; l=1,2\cdots,n_2.
\end{equation*}
\end{enumerate}
As the constraints of player 1 are linear and does not depend on the strategies of player 2, the strong Slater condition 
is not needed for the constraints of player 1. 

We use the following notations throughout this section. 
For $i=1,2$, $s\in S$, $l=1,2,\cdots,n_2$,
\begin{enumerate}[$\bullet$]
\item $\boldsymbol{C}^i(s)=\left[c^i(s,a^1,a^2)\right]_{a^1=1,a^2=1}^{|A^1(s)|,|A^2(s)|}$.
\item $\boldsymbol{C}^i=\mbox{diag}\left(\boldsymbol{C}^i(1),\boldsymbol{C}^i(2),\cdots,\boldsymbol{C}^i(|S|)\right)$.
\item $\boldsymbol{D}^{2,l}(s)= \left[d^{2,l}(s,a^1,a^2)\right]_{a^1=1,a^2=1}^{|A^1(s)|,|A^2(s)|}$. 
\item $x=\left(x\left(1\right)^T,x\left(2\right)^T,\cdots,x\left(|S|\right)^T\right)^T$.
\item $x(s)=\left(x(s,1),x(s,2),\cdots,x\left(s,|A^2(s)|\right)\right)^T$.
\item $u=\left(u(1),u(2),\cdots,u(|S|)\right)^T$.
\item $v\in \mathbb{R}$.
\item $z=\left(z(1),z(2),\cdots,z(|S|)\right)^T$.
\item $\delta^1=(\delta_1^1,\delta_2^1,\cdots,\delta_{n_1}^1)^T$.
\item $\delta^2=(\delta_1^2,\delta_2^2,\cdots,\delta_{n_2}^2)^T$.
\item $\textbf{1}_n=(1,1,\cdots,1)^T\in \mathbb{R}^n$.
\end{enumerate} 

\subsection{Single controller constrained stochastic game with average cost criterion}
In this section we consider the game described in Section \ref{SGCT} with average cost criterion where both players choose their strategies independently and minimize their expected average costs as defined in \eqref{SGCT_avg_cost} subject to their constraints from \eqref{SGCTsub}, \eqref{SGCT_avg_const2}. The constraints of player 1 given in \eqref{SGCTsub} are subscription based.
The expected average constraints  \eqref{SGCT_avg_const2} of player 2 captures the fact that the average consumption of resource $l$, $l=1,2,\cdots,n_2$, by player 2 is not more than given $\xi_l^2$.
\subsubsection{Average occupation measure}
For an initial distribution $\gamma$ and a stationary strategy $g$ define the average occupation measure
\[
 \pi^2_{ea}(\gamma,g):= \left\{\pi_{ea}^2(\gamma,g;s,a^2): s\in S, a^2\in A^2(s)\right\}.
\]
For all $s\in S$, $a^2\in A^2(s)$, $\pi_{ea}^2(\gamma,g;s,a^2)$ is given by
\begin{equation}\label{SGCT_avg_occ_msure}
 \pi_{ea}^2(\gamma,g;s,a^2)=\pi^g(s)g(s,a^2)
\end{equation}
where $\pi^g=\left(\pi^g(1),\pi^g(2),\cdots,\pi^g(|S|)\right)$ is steady state distribution of Markov chain induced by stationary strategy $g$ which exists and is unique under (A1).
$\pi^2_{ea}(\gamma,g)$ can be considered as a probability measure over $\mathcal{K}^2$ that assigns probability $\pi_{ea}^2(\gamma,g;s,a^2)$ to the 
state-action pair $(s,a^2)$. 
The occupation measure defined as in \eqref{SGCT_avg_occ_msure} is independent from initial distribution $\gamma$, so, we drop $\gamma$ from the notation. 
For fixed strategy pair $(f,g)\in F_S\times G_S$ 
the expected average costs of both the players can be written in terms of occupation measure as
 \begin{equation*}\label{SGCT_totalavgcost}
C_{ea}^i(f,g)= \sum_{(s,a^2)\in \mathcal{K}^2} \pi_{ea}^2(g;s,a^2)\sum_{a^1\in A^1(s)}f(s,a^1)c^i(s,a^1,a^2),\;\;\forall\;i=1,2.
\end{equation*}

\begin{equation*}\label{SGCT_avg_constraints2}
D_{ea}^{2,l}(f,g)= \sum_{(s,a^2)\in \mathcal{K}^2} \pi_{ea}^2(g;s,a^2)\sum_{a^1\in A^1(s)}f(s,a^1)
d^{2,l}(s,a^1,a^2)
\end{equation*}
for all $l=1,2,\cdots,n_2$.

Let $Q_{ea}$ be the set of vectors $x\in \mathbb{R}^{|\mathcal{K}^2|}$ satisfying
\begin{equation*}\label{SGCT_avg_ach_occmsure}
\left\{
\begin{aligned}
&(i) ~ \sum_{(s,a^2)\in \mathcal{K}^2}\left(\delta(s,s')-p(s'|s,a^2)\right)x(s,a^2)= 0,\;\; \forall\; s'\in S \\
&(ii) ~ \sum_{(s,a^2)\in \mathcal{K}^2}x(s,a^2)=1\\
&(iii) ~ x(s,a^2)\geq 0, \;\; \forall\; s\in S,\; a^2\in A^2(s).\\
\end{aligned}
\right. 
\end{equation*}
$\delta(\cdot,\cdot)$ is a Kronecker delta, i.e.,
\begin{equation*}
\delta(s,s')= 
\begin{cases} 1 & \text{if $s=s'$,}
\\
0 &\text{if $s\neq s'$.}
\end{cases}
\end{equation*}
The stationary strategies are complete, i.e., set of occupation measures achieved by history dependent strategies equals to those achieved by stationary strategies and further equals to the set $Q_{ea}$ \cite{Altman2}. 
It is known that for each $(s,a^2)\in~\mathcal{K}^2$, 
$x(s,a^2)= \pi_{ea}^2(g;s,a^2)$ where 
\begin{equation}\label{SGCT_avg_2nd_plstrg}
 g(s,a^2)=\frac{x(s,a^2)}{\sum_{a^2\in A^2(s)}x(s,a^2)}
\end{equation}
whenever denominator is nonzero (when it is zero $g(s)$ is chosen arbitrarily from $\wp(A^2(s))$) \cite{Altman2}.

The cost of player 1 when he uses action $a^1$ at state $s$ and player 2 uses strategy $g$ is given by
\[
 c^1(s,a^1;g)=\sum_{(s,a^2)\in\mathcal{K}^2} c^1(s,a^1,a^2)\pi_{ea}^2(g;s,a^2).
\]
Similarly, the costs of player 2 when he uses action $a^2$ at state $s$ and player 1 uses strategy $f$ are given by
\begin{gather*}
 c^2(f;s,a^2)=\sum_{a^1\in A^1(s)} c^2(s,a^1,a^2)f(s,a^1).\\
d^{2,l}(f;s,a^2)=\sum_{a^1\in A^1(s)} d^{2,l}(s,a^1,a^2)f(s,a^1),\;\; \forall\; l=1,2,\cdots,n_2.
\end{gather*}

\subsubsection{Mathematical programming formulation}\label{SGCT_mathprgm}
We show the one to one correspondence between the stationary Nash equilibria of 
single controller constrained stochastic game $G^c$ with average cost criterion
and the global minima of a certain mathematical program.

\subsubsection*{Best response linear programs}
For a given stationary strategy of one player in a two player constrained stochastic game, the best response of the other player is given by solving a constrained Markov decision model, which, in turn, can be obtained by a linear program in finite state-action setting \cite{Altman2}. For fixed strategy $g$ of player 2, the best response of player 1 can be obtained from the following linear program:
\begin{equation}\label{SGCT_avg_brp1}
 \left.
\begin{aligned}
 & \min_f \sum_{(s,a^1)\in\mathcal{K}^1} c^1(s,a^1;g)f(s,a^1)\\
\text{s.t.}\\
& (i) ~ \sum_{(s,a^1)\in\mathcal{K}^1}d_{sub}^{1,k}(s,a^1)f(s,a^1)\leq \xi_k^1, \; \; \forall \; \; k=1,2,\cdots,n_1\\
&(ii) ~ \sum_{a^1\in A^1(s)} f(s,a^1)=1, \; \; \forall \; \; s\in S \\
& (iii) ~ f(s,a^1)\geq 0, \; \; \forall \; \; s\in S, \; a^1\in A^1(s).
\end{aligned}
\right\}
\end{equation}
The dual of \eqref{SGCT_avg_brp1} is 
\begin{equation}\label{SGCT_avg_dual1}
\left.
\begin{aligned}
&\max_{z, \; \delta^1}\left[\sum_{s\in S} z(s)-\sum_{k=1}^{n_1}\delta_k^1 \xi_k^1\right]\\
\text{s.t.}\\
& (i) ~ z(s)\leq c^1(s,a^1;g)+\sum_{k=1}^{n_1}\delta_k^1 d^{1,k}_{sub}(s,a^1), \; \; \forall \; \; s\in S, \; a^1\in A^1(s)\\
&(ii) ~ \delta_k^1\geq 0, \; \; \forall \; \; k=1,2,\cdots,n_1.
\end{aligned}
\right\}
\end{equation}
Similarly, for fixed strategy $f$ of player 1, the best response of player 2 can be obtained from the 
following linear program:
\begin{equation}\label{SGCT_avg_brp2}
 \left.
\begin{aligned}
&\min_x \sum_{(s,a^2)\in \mathcal{K}^2}c^2(f;s,a^2)x(s,a^2)\\
\text{s.t.}\\
&(i) ~ \sum_{(s,a^2)\in \mathcal{K}^2}\left(\delta(s,s')- p(s'|s,a^2)\right)x(s,a^2)= 0, \; \; \forall \; \; s'\in S\\
&(ii) ~ \sum_{(s,a^2)\in \mathcal{K}^2}x(s,a^2)=1\\
&(iii)~ \sum_{(s,a^2)\in \mathcal{K}^2}d^{2,l}(f;s,a^2)x(s,a^2)\leq \xi_l^2, \; \; \forall \; \; l=1,2,\cdots,n_2\\
& (iv) ~ x(s,a^2)\geq 0, \; \; \forall \; \; s\in S, \; a^2\in A^2(s). 
\end{aligned}
\right\}
\end{equation}
If $x^*$ is the optimal solution of the linear program \eqref{SGCT_avg_brp2}, then the best response 
strategy $g^*$ of player 2  can be obtained from \eqref{SGCT_avg_2nd_plstrg} \cite{Altman2}.  The dual of the 
linear program \eqref{SGCT_avg_brp2} is given by 
\begin{equation}\label{SGCT_avg_dual2}
\left.
\begin{aligned}
& \max_{v,u \; \delta^2}\left[v-\sum_{l=1}^{n_2} \delta_l^2\xi_l^2\right]\\
\text{s.t.}\\
&(i) ~ v+u(s)\leq c^2(f;s,a^2)+\sum_{l=1}^{n_2}\delta_l^2 d^{2,l}(f;s,a^2)\\
&\hspace{3cm}+\sum_{s'\in S} p(s'|s,a^2)u(s'), \; \; \forall \; \; s\in S,\; a^2\in A^2(s)\\
&(ii) ~ \delta_l^2\geq 0,   \; \; \forall \; \; l=1,2,\cdots,n_2.
\end{aligned}
\right \}
\end{equation}

We denote the decision variables and objective function of mathematical program [MP1] by $\eta=(v,u^T,z^T,f^T,x^T,(\delta^1)^T,(\delta^2)^T)^T$ and $\Phi(\eta)$  respectively. 

\begin{theorem}\label{SGCTMain_avg_thm}
\begin{enumerate}
\item [(a)] If $(f^*,g^*)$ is a Nash 
equilibrium of the constrained stochastic game $G^c$ with average cost criterion, then, there exists a vector 
$\eta^*=\left(v^*,u^{*T},z^{*T},f^{*T},x^{*T},(\delta^{1*})^T,(\delta^{2*})^T\right)^T$ such that it is 
a global minimum of mathematical program \textup{{[MP1]}} given below 

{\allowdisplaybreaks
\begin{align*}
&\textup{{[MP1]}} \quad   \min_{\eta}\left[\left(f^T\boldsymbol{C}^1x-\left(\textbf{1}_{|S|}^T z -(\delta^1)^T\xi^1\right)\right)
+\left(f^T\boldsymbol{C}^2x-\left(v-(\delta^2)^T\xi^2\right)\right)\right]  \\
&\text{s.t.} \\
& (i) ~ v+ u(s)\leq \left[\left(f(s)\right)^T\boldsymbol{C}^2(s)\right]_{a^2}
+\sum_{l=1}^{n_2}\delta^2_l \left[\left(f(s)\right)^T\boldsymbol{D}^{2,l}(s)\right]_{a^2}  \\
&\hspace{3cm}+\sum_{s'\in S} p(s'|s,a^2)u(s'), \; \; \forall \;\; s\in S,\; a^2\in A^2(s) \\
& (ii) ~ z(s)\leq \left[\boldsymbol{C}^1(s)x(s)\right]_{a^1}+\sum_{k=1}^{n_1}\delta_k^1 d^{1,k}_{sub}(s,a^1),\; \; \forall \;\; s\in S,\; a^1\in A^1(s)\\
& (iii)~\sum_{(s,a^2)\in \mathcal{K}^2}\left[\delta(s,s')- p(s'|s,a^2)\right]x(s,a^2)=0, \;\;\forall \;\; s'\in S \\
& (iv) ~ \sum_{(s,a^2)\in \mathcal{K}^2}x(s,a^2)=1 \\
& (v) ~ \sum_{(s,a^1)\in \mathcal{K}^1}d^{1,k}_{sub}(s,a^1)f(s,a^1)\leq \xi_k^1, \;\;\forall \;\; k=1,2,\cdots,n_1 \\
& (vi)  ~ \sum_{s\in S} (f(s))^T\boldsymbol{D}^{2,l}(s)x(s)\leq \xi_l^2, \;\;\forall \;\; l=1,2,\cdots,n_2\\
& (vii) ~ \sum_{a^1\in A^1(s)} f(s,a^1)=1,  \;\;\forall \;\; s\in S \\
& (viii) ~ f(s,a^1)\geq 0, \;\;\forall \;\; s\in S,\;a^1\in A^1(s)  \\
& (ix) ~   x(s,a^2)\geq 0, \;\; \forall \;\; s\in S,\;a^2\in A^2(s) \\
& (x)  ~ \delta_k^1\geq 0, \;\; \forall \;\; k=1,2,\cdots,n_1 \\
& (xi) ~ \delta^2_l\geq 0,  \;\; \forall \;\; l=1,2,\cdots,n_2.
\end{align*}}
with $\Phi(\eta^*)=0$. 
\item [(b)] If $\eta^*=\left(v^*,u^{*T},z^{*T},f^{*T},x^{*T},(\delta^{1*})^T,(\delta^{2*})^T\right)^T$ 
is a global minimum of \textup{{[MP1]}} with $\Phi(\eta^*)=0$, then, $(f^*,g^{*})$ is a Nash equilibrium 
of the constrained stochastic game $G^c$ with average cost criterion, where
\[
 g^*(s,a^2)=\frac{x^*(s,a^2)}{\sum_{a^2\in A^2(s)}x^*(s,a^2)}
\]
for all $s\in S$, $a^2\in A^2(s)$ whenever the denominator is non-zero (when it is zero $g^*(s)$ is chosen arbitrarily from  $\wp(A^2(s))$).
\end{enumerate}
\end{theorem}
\begin{proof}
$(a)$~ Let $(f^*,g^*)$ be Nash equilibrium of the constrained stochastic game $G^c$ with average cost criterion. We  
construct occupation measure
$x^*$ corresponding to $g^*$ as given in \eqref{SGCT_avg_occ_msure} then $x^*$
satisfies $(iii)$, $(iv)$ and $(ix)$ of [MP1]. The strategy pair $(f^*,g^*)$
is feasible because it is a Nash equilibrium, so, $(f^*,x^*)$ satisfy $(v)$-$(viii)$ of [MP1].
As  $f^*$ and $g^*$ are best responses of each other, $x^*$ as constructed above will be optimal solution
of linear program \eqref{SGCT_avg_brp2} for fixed $f^*$ from Proposition $3.1(ii)$ of \cite{Shwartz}. 
By strong duality theorem \cite{Bertsimas}, \cite{Bazaraa}
there exists optimal solution ($v^*$, $u^*$, $\delta^{2*}$) of \eqref{SGCT_avg_dual2} such that $(v^*,u^*,f^*,\delta^{2*})$ satisfy $(i)$ and $(xi)$ of [MP1] and objective function value of \eqref{SGCT_avg_brp2} and \eqref{SGCT_avg_dual2} are equal. 
Similarly, $f^*$ is an optimal solution of linear program
\eqref{SGCT_avg_brp1} for fixed $g^*$ and hence 
there exists optimal solution ($z^*$, $\delta^{1*}$) of \eqref{SGCT_avg_dual1} such that $(z^*,x^*,\delta^{1*})$ satisfy $(ii)$ and $(x)$ of 
[MP1] and objective function value of \eqref{SGCT_avg_brp1} and \eqref{SGCT_avg_dual1}
are equal. In other words we have a point $\eta^{*}=\left(v^*,u^{*T},z^{*T},f^{*T},x^{*T},(\delta^{1*})^T,(\delta^{2*})^T\right)^T$ such that $(i), (ii), (x) \; \mbox{and} \;  (xi)$ are satisfied and 
\[
 f^{*T} \boldsymbol{C}^1 x^*=\textbf{1}_{|S|}^T z^*-(\delta^{1*})^T\xi^1,
\]
\[
 f^{*T} \boldsymbol{C}^2 x^*= v^* -(\delta^{2*})^T\xi^2.
\]
Thus, $\eta^*$ is a feasible point of the mathematical program [MP1] and from the construction of 
the objective function, $\Phi(\eta^*)=0$.

 Let $\eta$ be any feasible point of [MP1]. Multiply each constraint in $(ii)$ 
of [MP1] corresponding to pair $(s,a^1)$ by $f(s,a^1)$ and then add
over all $(s,a^1)\in~\mathcal{K}^1$ and by using the constraints $(v)$, $(vii)$, $(viii)$ and $(x)$ we have
\begin{equation}\label{SGCT_fes1}
 f^T\boldsymbol{C}^1 x\geq \textbf{1}_{|S|}^T z-(\delta^1)^{T}\xi^1.
\end{equation} 
By using the similar arguments as above, i.e., multiply each constraint in $(i)$ of [MP1] corresponding 
to pair $(s,a^2)$ by $x(s,a^2)$ and add over all $(s,a^2)\in~\mathcal{K}^2$
and by using the constraints $(iii)$, $(iv)$, $(vi)$, 
$(ix)$ and $(xi)$, we have 
\begin{equation} \label{SGCT_fes2}
 f^T\boldsymbol{C}^2 x\geq v-(\delta^2)^T\xi^2.
\end{equation} 
We have from \eqref{SGCT_fes1} and \eqref{SGCT_fes2},  $\Phi(\eta)\geq 0$ for all feasible points $\eta$ of [MP1]. 
Thus $\eta^*$ is a global minimum of the [MP1].

$(b)$ ~ Let $\eta^*$ be a global minimum of [MP1] such that $\Phi(\eta^*)=0$.
As $\eta^*$ is a feasible point of [MP1] then \eqref{SGCT_fes1} and \eqref{SGCT_fes2} will also 
hold for $\eta^*$, i.e.,   
\begin{gather*}
 f^{*T}\boldsymbol{C}^1 x^*\geq \textbf{1}_{|S|}^T z^*-(\delta^{1*})^{T}\xi^1\\
 f^{*T}\boldsymbol{C}^2 x^*\geq v^*-(\delta^{2*})^{T}\xi^2.
\end{gather*}
From above, both the terms of objective function are non-negative at $\eta^*$ but the objective function value is zero at $\eta^*$ which means both the terms are individually zero, i.e., 
 
\begin{equation}\label{SGCTe5}
\left.
\begin{aligned}
& f^{*T}\boldsymbol{C}^1 x^*= \textbf{1}_{|S|}^T z^*-(\delta^{1*})^{T}\xi^1\\
& f^{*T}\boldsymbol{C}^2 x^*= v^*-(\delta^{2*})^{T}\xi^2.
\end{aligned}
\right\}
\end{equation}
Fix $\eta^*$, and from the same argument used as in \eqref{SGCT_fes1}
and by using the constraints $(v)$, $(vii)$, $(viii)$, $(x)$ and \eqref{SGCTe5} 
we have the following inequality
\[
 f^{*T} \boldsymbol{C}^1 x^*\leq f^T \boldsymbol{C}^1 x^*, \; \; \forall \; \mbox{$1$-feasible} \; \; (f,x^*),
\]
Similarly we have 
\[
f^{*T} \boldsymbol{C}^2 x^*\leq f^{*T} \boldsymbol{C}^2 x, \; \; \forall \; \mbox{2-feasible} \; (f^*,x) 
\]
In other words we can say that    
\begin{gather*}
 C_{ea}^1(f^*,g^*)\leq C_{ea}^1(f,g^*), \; \; \; \forall \; \mbox{1-feasible} (f,g^*)\\
C_{ea}^2(f^*,g^*)\leq C_{ea}^2(f^*,g), \; \; \; \forall \; \mbox{2-feasible} (f^*,g),
\end{gather*}
where 
\[
 g^*(s,a^2)=\frac{x^*(s,a^2)}{\sum_{a^2\in A^2(s)}x^*(s,a^2)}
\]
for all $s\in S$, $a^2\in A^2(s)$ whenever the denominator is non-zero (when it is zero $g^*(s)$ is chosen arbitrarily from $\wp(A^2(s))$).
This implies that $(f^*,g^*)$ is a Nash equilibrium of the
constrained stochastic game $G^c$ with average cost criterion.
\end{proof}
\begin{remark}
Because the diagonal elements of the objective function's Hessian matrix are zero, it will 
have some positive as well as some negative eigenvalues. So, the objective function of [MP1] is 
a non-convex function. As there are some non-convex constraints, the feasible 
region is also not a convex set. So, [MP1] is a non-convex constrained optimization problem. 
\end{remark}

\subsubsection{Special cases}
We consider two special cases. First, we consider nonzero sum game as defined in Section \ref{SGCT} with average cost criterion where the constraints of player 2 do not depend on the strategies of  player 1. Next, we briefly consider the zero sum game as considered in \cite{Hordijk2}.
\subsubsection*{ (i) Quadratic program in the case of decoupled constraints}
We consider the situation where the constraints of player 2 do not depend on the strategies of the player 1.
This is possible when the immediate costs of player 2 which correspond to the constraints of player 2 do not depend 
on the actions of player 1, i.e., 
\begin{equation}\label{SGCT_LPeq} 
 d^{2,l}(s,a^1,a^2)=d^{2,l}(s,a^2), \; \forall \; s\in S, a^1\in A^1(s), a^2\in A^2(s)\; \mbox{and} \; \forall \; l=1,2,\cdots,n_2.
\end{equation} 
Under this condition [MP1] reduces to the quadratic program [QP1] given below:
{\allowdisplaybreaks
\begin{align*}
&\textup{{[QP1]}} \quad  \min_{\eta}\left[\left(f^T\boldsymbol{C}^1x-\left(\textbf{1}_{|S|}^T z -(\delta^1)^T\xi^1\right)\right)
+\left(f^T\boldsymbol{C}^2x-\left(v-(\delta^2)^T\xi^2\right)\right)\right] \\
&\text{s.t.}\\
& (i) ~ v+ u(s)\leq \left[(f(s))^T\boldsymbol{C}^2(s)\right]_{a^2}
+\sum_{l=1}^{n_2}\delta^2_l d^{2,l}(s,a^2)\\
&\hspace{3cm}+\sum_{s'\in S} p(s'|s,a^2)u(s'), \; \; \forall \;\; s\in S,\; a^2\in A^2(s)\\
& (ii) ~ z(s)\leq \left[\boldsymbol{C}^1(s)x(s)\right]_{a^1}+\sum_{k=1}^{n_1}\delta_k^1 d^{1,k}_{sub}(s,a^1),\; \; \forall \;\; s\in S,\; a^1\in A^1(s)\\
& (iii)~\sum_{(s,a^2)\in \mathcal{K}^2}\left[\delta(s,s')- p(s'|s,a^2)\right]x(s,a^2)=0, \;\;\forall \;\; s'\in S\\
& (iv) ~ \sum_{(s,a^2)\in \mathcal{K}^2}x(s,a^2)=1\\
& (v) ~ \sum_{(s,a^1)\in \mathcal{K}^1}d^{1,k}_{sub}(s,a^1)f(s,a^1)\leq \xi_k^1, \;\;\forall \;\; k=1,2,\cdots,n_1\\
& (vi)  ~ \sum_{(s,a^2)\in \mathcal{K}^2}d^{2,l}(s,a^2)x(s,a^2)\leq \xi_l^2, \;\;\forall \;\; l=1,2,\cdots,n_2\\
& (vii) ~ \sum_{a^1\in A^1(s)} f(s,a^1)=1,  \;\;\forall \;\; s\in S\\
& (viii) ~ f(s,a^1)\geq 0, \;\;\forall \;\; s\in S,\;a^1\in A^1(s) \\ 
& (ix) ~   x(s,a^2)\geq 0, \;\; \forall \;\; s\in S,\;a^2\in A^2(s)\\
& (x)  ~ \delta_k^1\geq 0, \;\; \forall \;\; k=1,2,\cdots,n_1\\
& (xi) ~ \delta^2_l\geq 0,  \;\; \forall \;\; l=1,2,\cdots,n_2.
\end{align*}
}
\subsubsection*{(ii) Zero sum single controller constrained stochastic games}  
The zero sum single controller constrained stochastic game with average cost criterion is considered in \cite{Hordijk2}.
We assume that player 1 minimizes the expected average cost of the game and player 2 has opposite objective,
 i.e., he  maximizes the expected average cost of the game.
In \cite{Hordijk2}, the player who controls the transition probabilities has realization based constraints and other player has no constraints and these games can be solved by a linear program. By substituting  $\boldsymbol{C}^1(s) =-\boldsymbol{C}^2(s)=\boldsymbol{C}(s)$ for all $s\in S$ and without the subscription type constraints, the quadratic  program [QP1] can be reduced into primal-dual pair of linear programs which are same as given in  \cite{Hordijk2}.

\subsection{Single controller constrained stochastic game with discounted cost criterion}
In this section we consider the game described in Section \ref{SGCT} with discounted cost criterion 
where both players choose their strategies independently and minimize their expected discounted costs as defined in \eqref{SGCT_cost}
subject to their constraints from \eqref{SGCTsub}, \eqref{SGCT_disc_const2}. The constraints of player 1 given in \eqref{SGCTsub} are subscription based. The expected discounted constraints \eqref{SGCT_disc_const2} of player 2 captures the fact that discounted cost for the consumption of resource $l$, $l=1,2,\cdots,n_2$, by player 2 is not more than given $\xi_l^2$.
Similar to the average cost criterion we give one mathematical program which characterizes stationary Nash equilibria of these games.
\subsubsection{Discounted occupation measure}
For an initial distribution $\gamma$ and a stationary strategy $g$ define the discounted occupation measure
\[
 \pi_\beta^2(\gamma,g):= \left\{\pi_\beta^2(\gamma,g;s,a^2): s\in S, a^2\in A^2(s)\right\}.
\]
For all $s\in S$, $a^2\in A^2(s)$, $\pi_\beta^2(\gamma,g;s,a^2)$ is  given by
\begin{equation}\label{SGCT_disc_occ_measure}
\pi_\beta^2(\gamma,g;s,a^2)= (1-\beta)\left(\sum_{t=0}^\infty \beta^t\sum_{s'\in S} \gamma(s')
 \left([P(g)]^t\right)_{s's}\right)g(s,a^2), 
\end{equation}
here $[P(g)]^0$ is the identity matrix. 
$\pi_\beta^2(\gamma,g)$ can be considered as a probability measure over $\mathcal{K}^2$ that assigns probability 
$\pi_\beta^2(\gamma,g;s,a^2)$ to the state-action pair $(s,a^2)$.
For fixed strategy pair $(f,g)\in F_S\times G_S$ 
the expected discounted costs of both players can be written in terms of occupation measure as
 \begin{equation*}\label{SGCT_total_disc_cost}
C_\beta^i(\gamma,f,g)= \sum_{(s,a^2)\in \mathcal{K}^2} \pi_\beta^2(\gamma,g;s,a^2)\sum_{a^1\in A^1(s)}f(s,a^1)c^i(s,a^1,a^2), \;\;\forall\; i=1,2.
\end{equation*}

\begin{equation*}\label{SGCT_disc_constraints2}
D_\beta^{2,l}(\gamma,f,g)= \sum_{(s,a^2)\in \mathcal{K}^2} \pi_\beta^2(\gamma,g;s,a^2)\sum_{a^1\in A^1(s)}f(s,a^1)
d^{2,l}(s,a^1,a^2)
\end{equation*}
for all $l=1,2,\cdots,n_2$.

Let $Q^\beta(\gamma)$ be the set of vectors $x\in \mathbb{R}^{|\mathcal{K}^2|}$ satisfying
\begin{equation*}\label{SGCT_disc_ach_occmsure}
\left\{
\begin{aligned}
&(i) ~ \sum_{(s,a^2)\in \mathcal{K}^2}\left(\delta(s,s')-\beta p(s'|s,a^2)\right)x(s,a^2)= (1-\beta)\gamma(s'),\;\; \forall\; s'\in S \\
&(ii) ~ x(s,a^2)\geq 0, \;\; \forall\; s\in S,\; a^2\in A^2(s).\\
\end{aligned}
\right. 
\end{equation*}
By summing the first constraint over $s'$ we note that $\sum_{(s,a^2)\in \mathcal{K}^2}x(s,a^2)=1$, so the $x$ satisfying the above constraints are probability measures.
The stationary strategies are complete, i.e., set of occupation measures achieved by history dependent strategies equals to those achieved by stationary strategies and further equals to the set $Q^\beta(\gamma)$ \cite{Altman2}. 
It is known that for each $(s,a^2)\in \mathcal{K}^2$, 
$x(s,a^2)= \pi_\beta^2(\gamma,g;s,a^2)$ where 
\begin{equation}\label{SGCT_disc_2nd_plstrg}
 g(s,a^2)=\frac{x(s,a^2)}{\sum_{a^2\in A^2(s)}x(s,a^2)}
\end{equation}
whenever denominator is nonzero (when it is zero $g(s)$ is chosen arbitrarily from $\wp(A^2(s))$) \cite{Altman2}.

The cost of player 1 when he uses action $a^1$ at state $s$ and player 2 uses strategy $g$ is given by
\[
 c^1(s,a^1;g)=\sum_{(s,a^2)\in\mathcal{K}^2} c^1(s,a^1,a^2)\pi_\beta^2(\gamma,g;s,a^2).
\]
Similarly, the costs of player 2 when he uses action $a^2$ at state $s$ and player 1 uses strategy $f$ are given by
\begin{gather*}
 c^2(f;s,a^2)=\sum_{a^1\in A^1(s)} c^2(s,a^1,a^2)f(s,a^1).\\
d^{2,l}(f;s,a^2)=\sum_{a^1\in A^1(s)} d^{2,l}(s,a^1,a^2)f(s,a^1),\;\; \forall\; l=1,2,\cdots,n_2.
\end{gather*}

\subsubsection{Mathematical programming formulation}
Similar to average cost criterion we show the one to one correspondence between the stationary Nash equilibria of this class of game and the global minima of a certain mathematical program.
\subsubsection*{Best response linear programs}
For fixed strategy $g$ of player 2, the best response of player 1 can be obtained from the following linear program:
\begin{equation}\label{SGCT_disc_brp1}
 \left.
\begin{aligned}
 & \min_f \sum_{(s,a^1)\in\mathcal{K}^1} c^1(s,a^1;g)f(s,a^1)\\
\text{s.t.}\\
& (i) ~ \sum_{(s,a^1)\in\mathcal{K}^1}d_{sub}^{1,k}(s,a^1)f(s,a^1)\leq \xi_k^1, \; \; \forall \; \; k=1,2,\cdots,n_1\\
&(ii) ~ \sum_{a^1\in A^1(s)} f(s,a^1)=1, \; \; \forall \; \; s\in S \\
& (iii) ~ f(s,a^1)\geq 0, \; \; \forall \; \; s\in S, \; a^1\in A^1(s).
\end{aligned}
\right\}
\end{equation}
The dual of \eqref{SGCT_disc_brp1} is 
\begin{equation}\label{SGCT_disc_dual1}
\left.
\begin{aligned}
&\max_{z, \; \delta^1}\left[\sum_{s\in S} z(s)-\sum_{k=1}^{n_1}\delta_k^1 \xi_k^1\right]\\
\text{s.t.}\\
& (i) ~ z(s)\leq c^1(s,a^1;g)+\sum_{k=1}^{n_1}\delta_k^1 d^{1,k}_{sub}(s,a^1), \; \; \forall \; \; s\in S, \; a^1\in A^1(s)\\
&(ii) ~ \delta_k^1\geq 0, \; \; \forall \; \; k=1,2,\cdots,n_1.
\end{aligned}
\right\}
\end{equation}
Similarly for fixed strategy $f$ of player 1, the best response of player 2 can be obtained from the 
following linear program:
\begin{equation}\label{SGCT_disc_brp2}
 \left.
\begin{aligned}
&\min_x \sum_{(s,a^2)\in \mathcal{K}^2}c^2(f;s,a^2)x(s,a^2)\\
\text{s.t.}\\
&(i) ~ \sum_{(s,a^2)\in \mathcal{K}^2}\left(\delta(s,s')-\beta p(s'|s,a^2)\right)x(s,a^2)= (1-\beta)\gamma(s'), \; \; \forall \; \; s'\in S\\
&(ii)~ \sum_{(s,a^2)\in \mathcal{K}^2}d^{2,l}(f;s,a^2)x(s,a^2)\leq \xi_l^2, \; \; \forall \; \; l=1,2,\cdots,n_2\\
& (iii) ~ x(s,a^2)\geq 0, \; \; \forall \; \; s\in S, \; a^2\in A^2(s). 
\end{aligned}
\right\}
\end{equation}
If $x^*$ is the optimal solution of the linear program \eqref{SGCT_disc_brp2} then the best response 
strategy $g^*$ of player 2  can be obtained from \eqref{SGCT_disc_2nd_plstrg} \cite{Altman2}. The dual of the 
linear program \eqref{SGCT_disc_brp2} is given by 
\begin{equation}\label{SGCT_disc_dual2}
\left.
\begin{aligned}
& \max_{u \; \delta^2}\left[\sum_{s\in S}(1-\beta)\gamma(s)u(s)
-\sum_{l=1}^{n_2} \delta_l^2\xi_l^2\right]\\
\text{s.t.}\\
&(i) ~ u(s)\leq c^2(f;s,a^2)+\sum_{l=1}^{n_2}\delta_l^2 d^{2,l}(f;s,a^2)\\
&\hspace{3cm}+\beta\sum_{s'\in S} p(s'|s,a^2)u(s'), \; \; \forall \; \; s\in S,\; a^2\in A^2(s)\\
&(ii) ~ \delta_l^2\geq 0,   \; \; \forall \; \; l=1,2,\cdots,n_2.
\end{aligned}
\right \}
\end{equation}
By using the best response linear programs \eqref{SGCT_disc_brp1}, \eqref{SGCT_disc_dual1}, \eqref{SGCT_disc_brp2},\eqref{SGCT_disc_dual2} we have similar results as in the case of average cost criterion. 
\begin{theorem}\label{SGCT_disc_Main_thm}
\begin{enumerate}
\item [(a)] If $(f^*,g^*)$ is a Nash 
equilibrium of the constrained stochastic game $G^c$ with discounted cost criterion, then, there exists a vector 
$\eta^*=\left(u^{*T},z^{*T},f^{*T},x^{*T},(\delta^{1*})^T,(\delta^{2*})^T\right)^T$ such that it is 
a global minimum of mathematical program \textup{{[MP2]}} given below  
{\allowdisplaybreaks
\begin{align*}
&\textup{{[MP2]}} \quad \min_{\eta}\Big[\left(f^T\boldsymbol{C}^1x-\left(\textbf{1}_{|S|}^T z -(\delta^1)^T\xi^1\right)\right)
 \\ & \hspace{3cm}+\left(f^T\boldsymbol{C}^2x-\left((1-\beta)\gamma^T u-(\delta^2)^T\xi^2\right)\right)\Big] \\
&\text{s.t.}\\
& (i) ~  u(s)\leq \left[(f(s))^T\boldsymbol{C}^2(s)\right]_{a^2}
+\sum_{l=1}^{n_2}\delta^2_l \left[(f(s))^T\boldsymbol{D}^{2,l}(s)\right]_{a^2}\\
&\hspace{3cm}+\beta\sum_{s'\in S} p(s'|s,a^2)u(s'), \; \; \forall \;\; s\in S,\; a^2\in A^2(s)\\
& (ii) ~ z(s)\leq \left[\boldsymbol{C}^1(s)x(s)\right]_{a^1}+\sum_{k=1}^{n_1}\delta_k^1 d^{1,k}_{sub}(s,a^1),\; \; \forall \;\; s\in S,\; a^1\in A^1(s)\\
& (iii)~\sum_{(s,a^2)\in \mathcal{K}^2}\left[\delta(s,s')-\beta p(s'|s,a^2)\right]x(s,a^2)=(1-\beta)\gamma(s'), \;\;\forall \;\; s'\in S\\
& (iv) ~ \sum_{(s,a^1)\in \mathcal{K}^1}d^{1,k}_{sub}(s,a^1)f(s,a^1)\leq \xi_k^1, \;\;\forall \;\; k=1,2,\cdots,n_1\\
& (v)  ~ \sum_{s\in S} (f(s))^T\boldsymbol{D}^{2,l}(s)x(s)\leq \xi_l^2, \;\;\forall \;\; l=1,2,\cdots,n_2\\
& (vi) ~ \sum_{a^1\in A^1(s)} f(s,a^1)=1,  \;\;\forall \;\; s\in S\\
& (vii) ~ f(s,a^1)\geq 0, \;\;\forall \;\; s\in S,\;a^1\in A^1(s) \\ 
& (viii) ~   x(s,a^2)\geq 0, \;\; \forall \;\; s\in S,\;a^2\in A^2(s)\\
& (ix)  ~ \delta_k^1\geq 0, \;\; \forall \;\; k=1,2,\cdots,n_1\\
& (x) ~ \delta^2_l\geq 0,  \;\; \forall \;\; l=1,2,\cdots,n_2.
\end{align*}
}
with $\Phi(\eta^*)=0$.

\item [(b)] If $\eta^*=\left(u^{*T},z^{*T},f^{*T},x^{*T},(\delta^{1*})^T,(\delta^{2*})^T\right)^T$ 
is a global minimum of \textup{{[MP2]}} with $\Phi(\eta^*)=0$, then, $(f^*,g^{*})$ is a Nash equilibrium 
of the constrained stochastic game $G^c$ with discounted cost criterion, where
\[
 g^*(s,a^2)=\frac{x^*(s,a^2)}{\sum_{a^2\in A^2(s)}x^*(s,a^2)}
\]
for all $s\in S$, $a^2\in A^2(s)$ whenever the denominator is non-zero (when it is zero $g^*(s)$ is chosen arbitrarily from $\wp(A^2(s))$).
\end{enumerate}
\end{theorem} 
\begin{proof}
 We can prove this by using the best response linear programs \eqref{SGCT_disc_brp1}, \eqref{SGCT_disc_dual1}, \eqref{SGCT_disc_brp2},\eqref{SGCT_disc_dual2} and with similar argument given in the proof of Theorem \ref{SGCTMain_avg_thm}. 
\end{proof}
\begin{remark}
 Similar to [MP1], 
[MP2] is also a non-convex constrained optimization problem. 
\end{remark}
\begin{remark}
Both [MP1] and [MP2] can be obtained from single mathematical program [MP4] given in Appendix (A). 
\end{remark}

\subsubsection{Special cases}
We consider two special cases. First, we consider nonzero sum game as defined in Section \ref{SGCT} with discounted cost criterion where the constraints of player 2 do not depend on the strategies of the player 1. Next, we briefly consider the zero sum game as considered in \cite{Altman}.
\subsubsection*{(i) Quadratic program in case of decoupled constraints}
When the constraints of player 2 do not depend on the strategies of player 1, i.e., under condition \eqref{SGCT_LPeq} the mathematical program [MP2] reduces to a quadratic program [QP2] given below
{\allowdisplaybreaks
\begin{align*}
&\textup{{[QP2]}} \quad \min_{\eta}\Big[\left(f^T\boldsymbol{C}^1x-\left(\textbf{1}_{|S|}^T z -(\delta^1)^T\xi^1\right)\right)
 \\ & \hspace{3cm}+\left(f^T\boldsymbol{C}^2x-\left((1-\beta)\gamma^T u-(\delta^2)^T\xi^2\right)\right)\Big] \\
&\text{s.t.}\\
& (i) ~  u(s)\leq \left[(f(s))^T\boldsymbol{C}^2(s)\right]_{a^2}
+\sum_{l=1}^{n_2}\delta^2_l d^{2,l}(s,a^2)\\
&\hspace{4cm}+\beta\sum_{s'\in S} p(s'|s,a^2)u(s'), \; \; \forall \;\; s\in S,\; a^2\in A^2(s)\\
& (ii) ~ z(s)\leq \left[\boldsymbol{C}^1(s)x(s)\right]_{a^1}+\sum_{k=1}^{n_1}\delta_k^1 d^{1,k}_{sub}(s,a^1),\;\forall\;s\in S, a^1\in A^1(s)\\
& (iii)~\sum_{(s,a^2)\in \mathcal{K}^2}\left[\delta(s,s')-\beta p(s'|s,a^2)\right]x(s,a^2)=(1-\beta)\gamma(s'), \;\;\forall \;\; s'\in S\\
& (iv) ~ \sum_{(s,a^1)\in \mathcal{K}^1}d^{1,k}_{sub}(s,a^1)f(s,a^1)\leq \xi_k^1, \;\;\forall \;\; k=1,2,\cdots,n_1\\
& (v)  ~ \sum_{(s,a^2)\in\mathcal{K}^2}d^{2,l}(s,a^2)x(s,a^2)\leq \xi_l^2, \;\;\forall \;\; l=1,2,\cdots,n_2\\
& (vi) ~ \sum_{a^1\in A^1(s)} f(s,a^1)=1,  \;\;\forall \;\; s\in S\\
& (vii) ~ f(s,a^1)\geq 0, \;\;\forall \;\; s\in S,\;a^1\in A^1(s) \\ 
& (viii) ~   x(s,a^2)\geq 0, \;\; \forall \;\; s\in S,\;a^2\in A^2(s)\\
& (ix)  ~ \delta_k^1\geq 0, \;\; \forall \;\; k=1,2,\cdots,n_1\\
& (x) ~ \delta^2_l\geq 0,  \;\; \forall \;\; l=1,2,\cdots,n_2.
\end{align*}
}
\subsubsection*{(ii) Zero sum single controller constrained stochastic games} 
The zero sum single controller constrained stochastic game with discounted cost criterion is considered in \cite{Altman}. In \cite{Altman}, the first player has subscription based constraints and second player has realization based constraints which do not depend on the strategies of first player and these games can be solved by a linear program. Setting $\boldsymbol{C}^1(s) =-\boldsymbol{C}^2(s)=\boldsymbol{C}(s)$  for all $s\in S$ the quadratic program [QP2] can be separated into primal-dual pair of linear programs which are same as given in \cite{Altman}.

\subsection{A Numerical Example}\label{SGCT_num}
We give one numerical example where immediate costs of player 2 corresponding to his constraints
do not depend on the actions of player 1. We compute the Nash equilibrium of this game  by solving corresponding quadratic program.
The components of the stochastic game are 
\begin{enumerate}
\item The state space $S=\{1,2\}$. 
\item The action sets of both the players are $A^i(s)=\{1,2\}$, $i=1,2$, $s=1,2$.
\item The immediate costs of both the players that defines their expected cost which
they want to minimize and transition probabilities of the game are given in the Table \ref{SGCT_CTP1} 
and \ref{SGCT_CTP2}. 
\begin{center}
\begin{table}[ht]
\caption{Immediate costs and transition probabilities} 
\centering
\begin{minipage}[]{5cm}
\subtable[$s=1$]{\label{SGCT_CTP1}
\begin{tabular}{|l|r|}
\hline
\backslashbox{(5,4)}{($\frac{1}{2}$, $\frac{1}{2}$)} & \backslashbox{(6,3)} {($\frac{1}{3}$,$\frac{2}{3}$)}\\
\hline
\backslashbox{(7,3)}{($\frac{1}{2}$, $\frac{1}{2}$)} &\backslashbox{(4,6)}{($\frac{1}{3}$,$\frac{2}{3}$)}\\
\hline
\end{tabular}
}
\end{minipage} 
\hspace{1cm}
\begin{minipage}[]{5cm}
\subtable[$s=2$]{\label{SGCT_CTP2}
\begin{tabular}{|l|r|}
\hline
\backslashbox{(2,3)}{(1, 0)} & \backslashbox{(4,2)} {($\frac{1}{5}$,$\frac{4}{5}$)}\\
\hline
\backslashbox{(3,1)}{(1,0)} &\backslashbox{(3,4)}{($\frac{1}{5}$,$\frac{4}{5}$)}\\
\hline
\end{tabular}
}
\end{minipage}
\end{table}
\end{center}
In both the tables above, the entry in upper triangle in each box gives the transition
probabilities and the entry in lower triangle gives the immediate cost of both the
players corresponding to the actions chosen by both the players in that state. For
example, at state 1 when both the player choose action $1$, then, player 1 gets immediate
cost 5 and player 2 gets 4 and this is represented by entry $(5, 4)$ and game will
remain in state 1 with probability $\frac{1}{2}$ and it can move to state 2 with probability $\frac{1}{2}$ and this is 
represented by entry $(\frac{1}{2}, \frac{1}{2})$ in the table corresponding to state 1. It is easy to check that transition probabilities given in tables above satisfies the ergodicity assumption (A1).

\item Both the players have one constraint, i.e., player 1 has one subscription based constraint and player 2 has 
one realization based constraint. The subscription cost of player 1 and immediate cost of player 2 
corresponding to each state-action pair are given in Table \ref{SGCT_sub_const} and \ref{SGCT_real_const} respectively.
\begin{table}[ht]%
\caption{Costs defining constraints}
\centering
\subtable [$d^1_{sub}(s,a^1)$]{\label{SGCT_sub_const}
\begin{tabular}{|c|c|c|}\hline
{} & $s=1$ & $s=2$ \\ \hline
$a^1=1$ & 2 & 3 \\ \hline
$a^1=2$ & 3 & 1 \\ \hline
\end{tabular}}
\hspace{1cm}
\subtable[$d^2(s,a^2)$]{\label{SGCT_real_const}
\begin{tabular}{|c|c|c|}\hline
{} & $s=1$ & $s=2$ \\ \hline
$a^2=1$ & 1 & 4 \\ \hline
$a^2=2$ & 2 & 5 \\ \hline
\end{tabular}}
\end{table}
\item The bound defining constraints are $\xi^1=4, \xi^2=2.5$.
\end{enumerate}
From Table \ref{SGCT_CTP1} and \ref{SGCT_CTP2} it is clear that the game is controlled by player 2.
\begin{itemize} 
\item[(i)]  For average cost criterion we solve the quadratic program [QP1] corresponding to the above data, by using MATLAB and obtain  
\begin{gather*}
 \eta^*=(3.0278,4.1667,2.833,3.8667,1.3067,0.6944,0.3056,0.3472,0.6528,\\
0.2667,0.36,0.3733,0,0.1867,0).
\end{gather*}
Note that at $\eta^*$ the objective function value is zero and hence it is the global minimum   
of quadratic program. We have
 $x^*(1,1)=0.2667$, $x^*(1,2)=0.36$, $x^*(2,1)=0.3733$,
$x^*(2,2)=0$. From \eqref{SGCT_avg_2nd_plstrg} we have $g^*(1,1)=0.4256$, $g^*(1,2)=0.5744$, $g^*(2,1)=1$, $g^*(2,2)=0$.
From Theorem \ref{SGCTMain_avg_thm}$(b)$ 
the Nash equilibrium of constrained stochastic game defined above with average cost criterion is
\begin{gather*}
f^*=\left((0.6944,0.3056),(0.3472,0.6528)\right), \;\; 
g^*=\left((0.4256,0.5744),(1,0)\right)
\end{gather*}
and the average costs of both the players at Nash equilibrium $(f^*,g^*)$ are
\begin{gather*}
 C_{ea}^1(f^*,g^*)= 4.4268\\
C_{ea}^2(f^*,g^*)= 3.0279.
\end{gather*}

\item [(ii)] For discounted cost criterion we take $\beta=0.5$, $\gamma=(0.5,0.5)$.
We solve the quadratic program [QP2] corresponding to the above data, by using MATLAB and obtain  
\begin{gather*}
 \eta^*=(10.2222,10.8888,3.5833, 1.4583,1,0,0.5,0.5,
0.3333,0.25,0.4167,0,\\0.2083,0.9444).
\end{gather*}
Note that at $\eta^*$ the objective function is zero and hence it is the global minimum of  
 quadratic program. 
We have $x^*(1,1)=0.3333$, $x^*(1,2)=0.25$, $x^*(2,1)=0.4167$, $x^*(2,2)=0$. From \eqref{SGCT_disc_2nd_plstrg} 
we have $g^*(1,1)=0.5714$, $g^*(1,2)=0.4286$, $g^*(2,1)=1$, $g^*(2,2)=0$. 
From Theorem \ref{SGCT_disc_Main_thm}$(b)$ 
the Nash equilibrium of constrained stochastic game defined above with discounted cost criterion is
\begin{gather*}
f^*=\left((1,0),(0.5,0.5)\right), \;\;
g^*=\left((0.5714,0.4286),(1,0)\right)
\end{gather*}
and the discounted costs of both the players at Nash equilibrium $(f^*,g^*)$ are
\begin{gather*}
 C_\beta^1(\gamma,f^*,g^*)= 4.2082\\
 C_\beta^2(\gamma,f^*,g^*)=  2.9166.
\end{gather*}
\end{itemize}

\section{Constrained stochastic game with independent state processes} \label{IND_model}
In this section we consider a $N$-player constrained stochastic game with independent state processes as discussed in \cite{Altman3}. In these games each player controls his own Markov chain, whose transition probabilities do not depend on the states and actions of other players. In these games at any time, each player has information only about current and past states of his Markov chain as well as of his previous actions and does not have any information about the states and actions of other players. However, each player wants to minimize his expected average cost that depends on the strategies of all the players. The expected average constraints of each player also depend on the strategies of all the players. These games come under the class of decentralized stochastic games.       
                                                  
The game is described by the tuple $\left(S^i,\gamma^i,A^i,c^i,d^i,p^i,\xi^i\right)$, $i=1,2,\cdots,N$, where: 
\begin{enumerate}
\item[(i)] $S^i$ is the finite state space of player $i$, $i = 1, \cdots, N$. The generic element of $S^i$ is denoted by $s^i$.  
Define, $S$:={\large \texttimes}$_{j=1}^N S^j$ and $S^{-i}$:={\large \texttimes}$_{j\neq i}S^j$ ({\large \texttimes} $\;$ stands for the product space). The element of $S$ is denoted by $s$ where $s=(s^1,s^2,\cdots,s^N)$ and $s^{-i}\in S^{-i}$ denote the vector of states $s^j$, $j\neq i$.

\item[(ii)] $\gamma^i$ is the probability distribution for the initial state of player $i$, $i = 1, \cdots, N$. We assume that the initial  states of all the players are independent. Denote $\gamma=(\gamma^1,\gamma^2,\cdots,\gamma^N)$.

\item[(iii)] $A^i$ is the finite action (strategy) set of player $i$ and its element is denoted by $a^i$, $i = 1, \cdots, N$.
 $A^i(s^i)$ denotes the set of all actions of player $i$ at state $s^i$ and 
  $A^i=\bigcup_{s^i\in S^i}A^i(s^i)$.
We denote $a=(a^1,a^2,\cdots,a^N)$ and $a^{-i}$ as vector of actions $a^j$, $j\neq i$. 

\item [(iv)] Define, $\mathcal{K}^i=\{(s^i,a^i)|s^i\in S^i, a^i\in A^i(s^i)\}$, $i=1,2,\cdots,N$ and 
$\mathcal{K}$={\large \texttimes}$_{i=1}^N \mathcal{K}^i$, $\mathcal{K}^{-i}$={\large \texttimes}$_{j\neq i} \mathcal{K}^j$.

\item [(v)] $c^i:\mathcal{K}\rightarrow\mathbb{R}$ is immediate cost of player $i$, $i = 1, \cdots, N$. Specifically,  
 $c^i(s,a)$ is the immediate cost incurred by player $i$, $i=1,2,\cdots,N$, when state of players is $(s^1, s^2,\cdots,s^N)$ and actions chosen by them are $(a^1,a^2,\cdots,a^N)$ respectively. Each player $i$, $i=1,2,\cdots,N$, wants to minimize the expected average cost involving $c^i(\cdot, \cdot)$.

\item [(vi)] $d^i=\left(d^{i,1},d^{i,2},\cdots,d^{i,n_i}\right)$, where $d^{i,k}:\mathcal{K}\rightarrow\mathbb{R}$ for all $k=1,2,\cdots,n_i$ are immediate costs of player $i$, $i = 1, \cdots, N$. These $d^{i,k}(\cdot, \cdot)$ are involved in the $k$th constraint, $k=1,2,\cdots,n_i$, on expected average cost of player $i$, $i = 1, \cdots, N$.

\item [(vii)] $p^i:\mathcal{K}^i\rightarrow \wp(S^i)$ is the transition probability of player $i$, $i = 1, \cdots, N$, where   $p^i(\bar{s}^i|s^i,a^i)$ is the probability that the 
state of player $i$ moves from state $s^i$ to $\bar{s}^i$ if he chooses action $a^i\in A^i(s^i)$.

\item [(viii)] $\xi^i=\left(\xi_1^i,\xi_2^i,\cdots,\xi_{n_i}^i\right)$ are the bounds defining the 
constraints of player $i$, $i = 1, \cdots, N$. 
\end{enumerate}

The game dynamics are as follows. Initially, at time  $t = 0$ state of the game is $ s = (s^1,s^2,\cdots,s^N)$ where $s^i\in S^i$ is chosen according to independent random variables $\gamma^i$, $i=1,2,\cdots,N$. Players independently choose actions $a = (a^1,a^2,\cdots,a^N)$  with $a^i\in A^i(s^i)$, $i=1,2,\cdots,N$. Player $i$ obtains an immediate cost $c^i(s,a)$, $i=1,2,\cdots, N$.
Apart from this cost, player $i$, $i=1,2,\cdots,N$, also receives  another $n_i$ costs $\{d^{i,k}(s,a)\}$, $k=1,2,\cdots,n_i$. 
These $\{d^{i,k}(\cdot, \cdot)\}$, $k=1,2,\cdots,n_i$, are involved in the  
expected average cost functionals of player $i$ which are constrained by specified bounds $\{\xi^i_k\}$, $k = 1, \cdots, n_i$.
Now, the state of player $i$ switches to a new state $\bar{s}^i$ at time $t=1$ with probability 
$p^i(\bar{s}^i|s^i,a^i)$, $i = 1, \cdots, N$. At time $t= 1$, in state $\bar{s}^i$, player $i$ then independently chooses an action $\bar{a}^i$, receives costs $c^i(\bar{s}, \bar{a})$ and $\{d^{i, k}(\bar{s}, \bar{a})\}$, $k = 1, \cdots, n_i$ and $i = 1, \cdots, N$. The next state for this player is $\tilde{s}^i$ with probability $p^i(\tilde{s}^i | \bar{s}^i, \bar{a}^i)$. 
The dynamics of the Markov chains repeat at new state $\tilde{s} = (\tilde{s}^1, \cdots, \tilde{s}^N)$ and game continues for infinite time horizon.  

While transition probabilities depend only on the present state and action used, actions that are used can depend on `past', as in history dependent strategies. Define a history of player $i$, $i=1,2,\cdots,N$, at time $t$ as 
$h_t^i=(s^i_0,a_0^i,s_1^i,a_1^i, \cdots, s^i_{t-1},a_{t-1}^i,s_t^i)$
where $s_t^i\in S^i$, $a_t^i\in A^i(s_t^i)$, $i=1,2,\cdots,N$, $t=0,1,2,\cdots$.
Let $H_t^i$ denote the set of all possible histories of length $t$ of player $i$.
Each player observes his own history and does not have any information about the other player's history.
A decision rule $f_t^i:H_t^i\rightarrow \wp(A^i(s_t^i))$ of player $i$ at time $t$ is a function which assigns to each history of length $t$ of player $i$, a probability measure over action set of player $i$. This means that under decision rule $f_t^i$ player $i$ chooses action $a^i$ with probability $f_t^i(h_t^i,a^i)$. 
The sequence of decision rules is called the strategy of the player.
Let $f^{ih}= (f_0^i,f_1^i,\cdots,f_t^i,\cdots)$ denote the strategy of player $i$, $i=1,2,\cdots,N$, and is called history dependent (behavioral) strategy. Note that the strategies of players do not depend on the 
realizations of the costs. If strategies were allowed to depend on such costs, then a player 
could use the costs to estimate the state and actions of the other players.  

Let $F^i$ denote the set of all history dependent strategies of player $i$ 
and $F$~={\large \texttimes}$_{i=1}^N F^i$ be the class of history dependent multi-strategies.
These strategies are called Markovian if at every decision epoch the decision rule depends 
only on the current state but the decision rule can differ at every epoch. 
A stationary strategy is a Markovian strategy which is independent of the time, i.e.,
at every decision epoch the decision rule is same. So, for stationary strategy 
$f_t^i=f^i$ for all $t$, i.e., $(f^i,f^i,f^i,\cdots)$ 
is a stationary strategy of player $i$. We denote, with 
some abuse of notations, $f^i$ as the stationary strategy of player $i$. 
Let $F_{S^i}$  denote the set of all stationary strategies of player $i$ and $F_S$={\large \texttimes}$_{i=1}^N F_{S^i}$ denote the class of stationary multi-strategies. For, $i=1,2,\cdots,N$, stationary strategy $f^i\in F_{S^i}$ is identified 
with $f^i=\left((f^i(1))^T,(f^i(2))^T,\cdots,(f^i(|S^i|))^T\right)^T$, where 
$f^i(s^i)=\left(f^i(s^i,1),f^i(s^i,2),\cdots,f^i\left(s^i,|A^i(s^i)|\right)\right)^T$
for all $s^i\in S^i$. For all $s^i\in S^i$, $f^i(s^i,a^i)$ is
then, the probability of choosing action $a^i\in A^i(s^i)$ by player $i$, $i = 1, \cdots, N$. 
For $f^h\in F$ we denote $f^{-ih}$ as the vector of strategies $f^{jh}$, $j\neq i$, and for any $g^{ih}\in F^i$ we define 
$(f^{-ih},g^{ih})$ to be the multi-strategy, where, for $j\neq i$, player $j$ uses $f^{jh}$ and player $i$ uses $g^{ih}$. 
Under mild assumptions, which we also make, Altman, {\it et al} \cite{Altman3} show that a Nash equilibrium exists for the above constrained stochastic game within the class of stationary strategies.

This leads to the introduction of vector stochastic process $\{X_t,\mathbb{A}_t\}_{t=0}^{ \infty}$, where
$X_t=\left(X_t^1,X_t^2,\cdots,X_t^N\right)$, $\mathbb{A}_t=\left(\mathbb{A}_t^1,\mathbb{A}_t^2,\cdots,\mathbb{A}_t^N\right)$, $X_t^i$ denote the state of the player $i$ and $\mathbb{A}_t^i$ denote the action chosen by player $i$ at time $t$, $t = 0, 1, \cdots$.
An initial distribution $\gamma$ together with multi-strategy $f^h\in F$  
defines a unique probability measure 
$\mathbb{P}_{f^h}^\gamma$ on an appropriate probability space with respect to which the laws of 
vector stochastic process
$\{X_t,\mathbb{A}_t\}_{t = 0}^{\infty}$ of states and actions can be defined. 
The expectation operator on this probability space is denoted by $\mathbb{E}_{f^h}^\gamma$.

\subsubsection*{The expected average costs}
These costs are average functionals of states and actions of  all the players and each player minimizes his cost functionals. 
For given initial distribution $\gamma$ and multi-strategy $f^h$ the expected average cost of player $i$, $i=1,2,\cdots,N$ 
is defined as
\begin{equation}\label{IND_cost}
 C_{ea}^i(\gamma,f^h)=\limsup_{T\rightarrow \infty}\frac{1}{T}\sum_{t=0}^{T-1}\mathbb{E}_{f^h}^\gamma 
c^i(X_t,\mathbb{A}_t).
\end{equation}

\subsubsection*{The expected average constraints}
The constraints of each player are defined by average functionals of states and actions of all the players which are bounded by given reals. For given initial distribution $\gamma$ and multi-strategy $f^h$ the expected average costs of player $i$, $i=1,2,\cdots N$ are defined as 
\begin{equation*}
 D_{ea}^{i,k}(\gamma,f^h)=\limsup_{T\rightarrow \infty}\frac{1}{T}\sum_{t=0}^{T-1}\mathbb{E}_{f^h}^\gamma 
d^{i,k}(X_t,\mathbb{A}_t), \;\; \forall \; k=1,2,\cdots,n_i.
\end{equation*}
$D_{ea}^{i,k}(\cdot,\cdot)$ can capture the average consumption of resource $k$, $k=1,2,\cdots,n_i$, by player $i$, $i=1,2,\cdots, N$. 
The constraints of player $i$, $i=1,2,\cdots,N$ are given as 
\begin{equation}\label{IND_real}
D_{ea}^{i,k}(\gamma,f^h)\leq \xi_k^i, \; \; \forall  \; k=1,2,\cdots,n_i.
\end{equation}
The constraints \eqref{IND_real} captures the fact that average consumption of resource $k$ by player $i$, $i=1,2,\cdots, N$, when player $i$ uses strategy $f^{ih}$ and other players use $f^{-ih}$ is not more than given $\xi_k^i$, $k=1,2,\cdots,n_i$.  

All the players choose their strategies independently and want to minimize their 
expected average cost from \eqref{IND_cost} subject to their constraints from \eqref{IND_real}. 
We denote this constrained stochastic game by $G_{ea}^c$.
The multi-strategy  $f^h=~(f^{1h},f^{2h},\cdots,f^{Nh})$ is called $i$-feasible if it satisfies $i$th player's constraints       
from \eqref{IND_real} and it is called feasible if it is $i$-feasible for every $i=1,2,\cdots,N$. 
Let $F^\xi$ denote the set of all feasible history dependent multi-strategies and 
$F^\xi_S$ denote the set of all stationary feasible multi-strategies for the constrained stochastic game $G_{ea}^c$. We shall assume throughout that $F_S^\xi$ is non-empty. 

We now recall the definition of Nash equilibrium as given in \cite{Altman3}. 
A multi-strategy $f^{h*}\in F^\xi$ is called  the Nash equilibrium of the constrained 
stochastic game $G_{ea}^c$, if for each player $i=1,2,\cdots,N$ and for any $f^{ih}$ such that $(f^{ih},f^{-ih*})$ is $i$-feasible, one has that 
\begin{gather*}\label{IND_NE_def}
C_{ea}^i(\gamma,f^{h*})\leq C_{ea}^i(\gamma,f^{ih},f^{-ih*}).
\end{gather*}
Thus, unilateral deviation by any player $i$, $i = 1, \cdots, N$ from equilibrium strategy $f^{h*}$ is not possible, because in that case, either at least one of his constraints will be violated or it will result in a cost for player $i$ that is not lower than the one achieved by feasible equilibrium strategy $f^{h*}$. 
A stationary multi-strategy $f^*\in F_S^\xi$ is said to be Nash equilibrium of constrained stochastic game $G_{ea}^c$, if 
for each player $i=1,2,\cdots,N$ and for any $f^{i}$ such that $(f^{i},f^{-i*})$ is $i$-feasible, one has that 
\begin{gather*}\label{IND_SNE_def}
C_{ea}^i(\gamma,f^{*})\leq C_{ea}^i(\gamma,f^{i},f^{-i*}).
\end{gather*}
 This can be seen by noticing that when all players $j$, $j\neq i$, fix their strategy as a stationary strategy then 
player $i$ is faced with a constrained Markov decision process (CMDP) where
optimal strategy always exists in the space of stationary strategies  \cite{Altman2}.

\subsubsection*{Assumptions [Altman, et al. \cite{Altman3}]}
As similar to \cite{Altman3} we also have the following assumptions:
\begin{enumerate}
\item[(A1)] Ergodicity: For each player $i$, $i = 1, \cdots, N$, and for any stationary strategy $f^i$ the state process
of player $i$ is an irreducible Markov chain with one ergodic class (and possibly some transient states).
\item[(A2)] Strong Slater condition: Every player $i$, $i = 1, \cdots, N$ has some strategy $g^i$ such that for any multi-strategy $f^{-i}$ of other players,
\begin{gather*}
D^{i,k}_{ea}(\gamma,(f^{-i},g^{i}))<\xi_k^i, \; \; \forall \; k=1,2,\cdots,n_i. 
\end{gather*}
\item[(A3)] The players do not observe their costs, i.e., the strategy chosen by any player does not depend on the realization of the cost.
\end{enumerate}
The last assumption is due to the definition of the strategies. If strategies were allowed to depend on the realization of the costs, then a player can use the cost to estimate other player's states and actions. 
As the Nash equilibrium exists in stationary strategies under the assumptions (A1)-(A3) \cite{Altman3}, from now onwards we restrict ourselves to the class of stationary strategies. 

\subsection{Average occupation measure}\label{IND_def_not}
For each player $i$, $i=1,2,\cdots,N$, using a stationary strategy $f^i$ and initial distribution $\gamma^i$ 
define the average occupation measure as
\begin{gather*}
\pi_{ea}^i(\gamma^i,f^i):= \left\{\pi_{ea}^i(\gamma^i,f^i;s^i,a^i): s^i\in S^i, a^i\in A^i(s^i)\right\}.
\end{gather*}
For all $s^i\in S^i$, $a^i\in A^i(s^i)$,
$\pi_{ea}^i(\gamma^i,f^i;s^i,a^i)$ is given by
\begin{equation}\label{IND_occ_measure}
\pi_{ea}^i(\gamma^i,f^i;s^i,a^i)= \pi^{f^i}(s^i)f^i(s^i,a^i) 
\end{equation}
where $\pi^{f^i}=\left(\pi^{f^i}(1),\pi^{f^i}(2),\cdots,\pi^{f^i}(|S^i|)\right)$ is the unique steady state distribution of Markov chain induced by strategy $f^i$ of player $i$,
which exists under (A1). $\pi_{ea}^i(\gamma^i,f^i)$ can be considered as probability measure over $\mathcal{K}^i$ that assigns probability $\pi_{ea}^i(\gamma^i,f^i;s^i,a^i)$ to state-action pair $(s^i,a^i)$. The occupation measure defined in \eqref{IND_occ_measure} is unique and independent from initial distribution $\gamma^i$, so, we drop $\gamma^i$ from the notation.
For any multi-strategy $f\in F_{S}$
the expected average costs for each player $i$, $i=1,2,\cdots N$, can be written in terms of occupation measure as
\begin{equation*}
C_{ea}^i(f)= \sum_{(s,a)\in \mathcal{K}}\left[\prod_{j=1}^N \pi_{ea}^j(f^j;s^j,a^j)\right]c^i(s,a).
\end{equation*}
\begin{equation*}
 D_{ea}^{i,k}(f)=\sum_{(s,a)\in \mathcal{K}}\left[\prod_{j=1}^N \pi_{ea}^j(f^j;s^j,a^j)\right]d^{i,k}(s,a),\;\;\forall\; k=1,2,\cdots,n_i. 
\end{equation*}
Let $Q_{ea}^i$, $i=1,2,\cdots,N$, be the set of vectors $x^i\in \mathbb{R}^{|\mathcal{K}^i|}$ satisfying
\begin{equation*}\label{IND_ach_occmsure}
\left\{ 
\begin{aligned}
&\sum_{(s^i,a^i)\in \mathcal{K}^i}\left(\delta(s^i,\bar{s}^i)- p^i(\bar{s}^i|s^i,a^i)\right)x^i(s^i,a^i)= 0,
\;\;\forall \; \bar{s}^i\in S^i  \\
&\sum_{(s^i,a^i)\in \mathcal{K}^i}x^i(s^i,a^i)=1\\
& x^i(s^i,a^i)\geq 0, \;\;\forall\; s^i\in S^i, a^i\in A^i(s^i). 
\end{aligned}
\right. 
\end{equation*}
The space of stationary strategies is complete, i.e., the set of occupation measures achieved by 
history dependent strategies equals to those achieved by stationary strategies and further equals to the set
$Q_{ea}^i$, $i=1,2,\cdots,N$ \cite{Altman2}. It is known that for each $(s^i,a^i)\in \mathcal{K}^i$, 
$x^i(s^i,a^i)= \pi_{ea}^i(f^i;s^i,a^i)$ where 
\begin{equation}\label{IND_plstrg}
 f^i(s^i,a^i)=\frac{x^i(s^i,a^i)}{\sum_{a^i\in A^i(s^i)}x^i(s^i,a^i)}
\end{equation}
whenever denominator is nonzero (when it is zero $f^i(s^i)$ is chosen arbitrarily from $\wp(A^i(s^i))$) \cite{Altman2}.

We use the following notations throughout this section.
For $i=1,2,\cdots,N$,  \\
 $\bullet ~ u^i=\left(u^i(1),u^i(2),\cdots,u^i\left(|S^i|\right)\right)^T$.\\
 $\bullet ~ v^i\in \mathbb{R}$. \\
$\bullet ~ x^i=\left(\left(x^i(1)\right)^T,\left(x^i(2)\right)^T,\cdots,\left(x^i\left(|S^i|\right)\right)^T\right)^T$.\\
 $\bullet ~ x^i(s^i)=\left(x^i(s^i,1),x^i(s^i,2),\cdots,x^i\left(s^i,|A^i(s^i)|\right)\right)^T$.\\ 
$\bullet ~ \delta^i=(\delta_1^i,\delta_2^i,\cdots,\delta_{n_i}^i)^T$.\\
The costs of player $i$, $i=1,2,\cdots, N$, when he uses action $a^i$ at state $s^i$ and other players use $f^{-i}$ is defined as in \cite{Altman3},
\begin{align*}
 c^i(f^{-i};s^i,a^i)&=\sum_{(s,a)^{-i}\in\mathcal{K}^{-i}}\left[\prod_{j=1;j\neq i}^{N}\pi_{ea}^j(f^j;s^j,a^j)\right]c^i(s,a). \\
d^{i,k}(f^{-i};s^i,a^i)&=\sum_{(s,a)^{-i}\in\mathcal{K}^{-i}}\left[\prod_{j=1;j\neq i}^{N}\pi_{ea}^j(f^j;s^j,a^j)\right]d^{i,k}(s,a),
\;\forall\; k=1,2,\cdots,n_i. 
\end{align*}
\subsection{Mathematical programming formulation}\label{IND_mathprgm}
We show the one to one correspondence between the stationary Nash equilibria of this game and the global minima of one 
mathematical program. 
\subsubsection*{Best response linear programs}
The best response of each player $i$, $i=1,2,\cdots,N$, against fixed stationary strategy $f^{-i}$ of other players 
is given by solving a constrained Markov decision model, 
which, in turn, can be obtained by a linear program in our setting \cite{Altman2}.
The best response of player $i$ against fixed strategy $f^{-i}$ of other players is given by the linear program below:
\begin{equation}\label{IND_brp}
 \left.
\begin{aligned}
& \min_{x^i}\sum_{(s^i,a^i)\in \mathcal{K}^i}c^i(f^{-i};s^i,a^i)x^i(s^i,a^i)\\
\text{s.t.}\\
&(i)~\sum_{(s^i,a^i)\in \mathcal{K}^i}\left(\delta(s^i,\bar{s}^i)- p^i(\bar{s}^i|s^i,a^i)\right)x^i(s^i,a^i)= 0, \;\; \forall\; \bar{s}^i\in S^i \\
&(ii)~\sum_{(s^i,a^i)\in \mathcal{K}^i}x^i(s^i,a^i)=1\\
&(iii)~\sum_{(s^i,a^i)\in \mathcal{K}^i}d^{i,k}(f^{-i};s^i,a^i)x^i(s^i,a^i)\leq \xi_k^i, \;\; \forall\; k=1,2,\cdots,n_i\\
&(iv)~ x^i(s^i,a^i)\geq 0, \;\; \forall\; s^i\in S^i, a^i\in A^i(s^i)
\end{aligned}
\right\}
\end{equation}
 If $x^{i*}$ is the optimal solution 
of the linear program \eqref{IND_brp}, then, the best response $f^{i*}$ of player $i$ can be obtained 
from \eqref{IND_plstrg} \cite{Altman2}.
The dual of linear program \eqref{IND_brp} is
\begin{equation}\label{IND_dual_brp}
 \left.
\begin{aligned}
&\max_{v^i,u^i,\delta^i} \left[v^i-\sum_{k=1}^{n_i}\delta^i_k\xi_k^i\right]\\
\text{s.t.}\\
&(i)~ v^i+u^i(s^i)\leq c^i(f^{-i};s^i,a^i)+\sum_{k=1}^{n_i}d^{i,k}(f^{-i};s^i,a^i)\delta^i_k\\
&\hspace{3cm}+\sum_{\bar{s}^i\in S^i}p^i(\bar{s}^i|s^i,a^i)u^i(\bar{s}^i),\;\;\forall\; s^i\in S^i, a^i\in A^i(s^i)\\
&(ii)~\delta_k^i\geq 0,\;\;\forall\; k=1,2,\cdots,n_i.\\ 
\end{aligned}
\right\}
\end{equation}
By using $N$ primal-dual pair of linear programs given by \eqref{IND_brp}, \eqref{IND_dual_brp}, we show the one to one correspondence between the stationary Nash equilibria of constrained stochastic game $G_{ea}^c$ and global minima of a mathematical program [MP3].   Let $\zeta^T:= (v^i,(u^i)^T,(x^i)^T,(\delta^i)^T)_{i=1}^N$ and $\psi(\zeta)$ denote the decision variables and the objective function of [MP3] respectively. $\zeta^T$ is a $1\times \big(N + \sum_{i=1}^N |S^i| + \sum_{i=1}^N \sum_{s^i\in S^i}|A^i(s^i)|+\sum_{i=1}^N n_i\big)$ dimensional vector.
\begin{theorem}\label{IND_Main_thm}
\begin{enumerate}
\item [(a)] If $(f^{i*})_{i=1}^N$ is a Nash 
equilibrium of the constrained stochastic game $G_{ea}^c$, then, there exists a vector 
$\zeta^{*T}=\left(v^{i*},(u^{i*})^T,(x^{i*})^T,(\delta^{i*})^T\right)_{i=1}^N$
such that it is a global minimum of mathematical program \textup{{[MP3]}} given below   
{\allowdisplaybreaks
\begin{align*}
&\textup{{[MP3]}} \quad \min_{\zeta}\sum_{i=1}^N\left[\sum_{(s,a)\in \mathcal{K}}\left(\prod_{j=1}^N x^j(s^j,a^j)\right)c^i(s,a) -\left(v^i-\sum_{k=1}^{n_i}\delta_k^i\xi_k^i\right)\right] \\ 
&\text{s.t.}\\
&(i) ~ v^i+u^i(s^i)\leq \sum_{(s,a)^{-i}\in\mathcal{K}^{-i}}\left(\prod_{j=1;j\neq i}^{N}x^j(s^j,a^j)\right)c^i(s,a)\\
&\hspace{2.2cm}+\sum_{k=1}^{n_i}\delta_k^i\left[\sum_{(s,a)^{-i}\in\mathcal{K}^{-i}}\left(\prod_{j=1;j\neq i}^{N}x^j(s^j,a^j)\right)
d^{i,k}(s,a)\right]\\
&\hspace{2.2cm}+\sum_{\bar{s}^i\in S^i}p^i(\bar{s}^i|s^i,a^i)u^i(\bar{s}^i),\;\forall\; s^i\in S^i, a^i\in A^i(s^i), i=1,2,\cdots,N\\
&(ii) ~ \sum_{(s^i,a^i)\in \mathcal{K}^i}\left(\delta(s^i,\bar{s}^i)- p^i(\bar{s}^i|s^i,a^i)\right)x^i(s^i,a^i)= 0, \;\forall\; \bar{s}^i\in S^i, i=1,2,\cdots,N\\
& (iii) ~ \sum_{(s^i,a^i)\in \mathcal{K}^i}x^i(s^i,a^i)=1,\;\;\forall\; i=1,2,\cdots,N\\
&(iv) ~ \sum_{(s,a)\in \mathcal{K}}\left(\prod_{j=1}^N x^j(s^j,a^j)\right)d^{i,k}(s,a)\leq \xi_k^i,\;\forall\; k=1,2,\cdots,n_i, i=1,2,\cdots,N\\
&(v) ~ x^i(s^i,a^i)\geq 0,\;\;\forall\; s^i\in S^i,\; a^i\in A^i(s^i), i=1,2,\cdots,N\\ 
&(vi) ~ \delta_k^i\geq 0,\;\;\forall\;  k=1,2,\cdots,n_i,\; i=1,2,\cdots,N.
\end{align*}
}
with $\psi(\zeta^*)=0$.

\item [(b)] If $\zeta^{*T}=\left(v^{i*},(u^{i*})^T,(x^{i*})^T,(\delta^{i*})^T\right)_{i=1}^N$
 is a global minimum of \textup{{[MP3]}} with $\psi(\zeta^*)=0$ then $(f^{i*})_{i=1}^N$ 
is a Nash equilibrium of the constrained stochastic game $G_{ea}^c$ where,
\begin{gather*}
 f^{i*}(s^i,a^i)=\frac{x^{i*}(s^i,a^i)}{\sum_{a^i\in A^i(s^i)}x^{i*}(s^i,a^i)}
\end{gather*}
for all  $s^i\in S^i$, $a^i\in A^i(s^i)$, $i=1,2,\cdots,N$ whenever the denominator is non-zero 
(when it is zero $f^{i*}(s^i)$ is chosen arbitrarily from $\wp(A^i(s^i))$).
\end{enumerate}
\end{theorem}
\begin{proof}
$(a)$ ~ Let $(f^{i*})_{i=1}^N$ be a Nash equilibrium of the constrained stochastic game $G_{ea}^c$. 
For each $i=1,2,\cdots,N$, we construct occupation measures $x^{i*}$ as in \eqref{IND_occ_measure}
corresponding to stationary strategies $f^{i*}$. Then, the constraints in $(ii), (iii)$ and $(v)$
are satisfied by $(x^{i*})_{i=1}^N$. The multi-strategy $(f^{i*})_{i=1}^N$ is feasible because it is a Nash equilibrium, 
so the constraints in $(iv)$ are also satisfied by $(x^{i*})_{i=1}^N$. For each $i=1,2,\cdots N$, $f^{i*}$ is best response 
of player $i$ against fixed strategy $f^{-i*}$ of other players; so, $x^{i*}$ as constructed above will be optimal solution
of linear program \eqref{IND_brp} for this fixed $f^{-i*}$ from Proposition $3.1(ii)$ of \cite{Altman3}. From strong duality theorem \cite{Bertsimas}, \cite{Bazaraa} there exist optimal solution $(v^{i*}, u^{i*}, \delta^{i*})$ of \eqref{IND_dual_brp}
such that the constraints in $(i)$ and $(vi)$ of [MP3] are satisfied by
$(v^{i*},u^{i*},x^{i*},\delta^{i*})_{i=1}^N$ and objective function value of \eqref{IND_brp} and \eqref{IND_dual_brp} are same.
In other words we have a point 
$\zeta^{*T}=\left(v^{i*},(u^{i*})^T,(x^{i*})^T,(\delta^{i*})^T\right)_{i=1}^N$ which 
is feasible for [MP3] and    
\[
 \sum_{(s,a)\in \mathcal{K}}\left(\prod_{j=1}^N x^{j*}(s^j,a^j)\right)c^i(s,a) = v^{i*}-\sum_{k=1}^{n_i}\delta_k^{i*}\xi_k^i,\;\;\forall\; i=1,2,\cdots,N.
\]
 From the construction of 
the objective function, $\psi(\zeta^*)=0$. 

Let $\zeta$ be any feasible point of [MP3]. 
For each $i=1,2,\cdots,N$, multiply
each constraint in $(i)$ of [MP3] corresponding to  
pair $(s^i,a^i)\in \mathcal{K}^i$ by 
$x^i(s^i,a^i)$,  add over all $(s^i,a^i)\in \mathcal{K}^i$ and by then using the constraints $(ii)$-$(vi)$ we have
\begin{equation}\label{IND_fes}
\sum_{(s,a)\in \mathcal{K}}\left(\prod_{j=1}^N x^{j}(s^j,a^j)\right)c^i(s,a)\geq v^{i}-\sum_{k=1}^{n_i}\delta_k^i\xi_k^i,\;\;\forall\; i=1,2,\cdots,N.
\end{equation}
From \eqref{IND_fes} we have $\psi(\zeta)\geq 0$ for all feasible points $\zeta$ of [MP3]. 
Thus $\zeta^*$ is a global minimum of the [MP3].
 
$(b)$ ~ Let $\zeta^*$ be a global minimum of [MP3] such that $\psi(\zeta^*)=0$. As $\zeta^*$ is 
a feasible point of [MP3] then \eqref{IND_fes} will also hold for $\zeta^*$, i.e.,
\[
 \sum_{(s,a)\in \mathcal{K}}\left(\prod_{j=1}^N x^{j*}(s^j,a^j)\right)c^i(s,a) \geq v^{i*}-\sum_{k=1}^{n_i}\delta_k^{i*}\xi_k^i,\;\;\forall\; i=1,2,\cdots,N.
\]
From above we see that all $N$ terms of the objective function are non-negative at $\zeta^*$. But at $\zeta^*$ the objective function value is zero which means that all the terms are individually zero, i.e., 
\begin{equation}\label{IND_NASH_eq} 
 \sum_{(s,a)\in \mathcal{K}}\left(\prod_{j=1}^N x^{j*}(s^j,a^j)\right)c^i(s,a) = v^{i*}-\sum_{k=1}^{n_i}\delta_k^{i*}\xi_k^i, \;\;\forall\; i=1,2,\cdots,N.
\end{equation}
Fix $\zeta^*$ and for each $i=1,2,\cdots,N$, multiply each constraint in $(i)$ corresponding to pair $(s^i,a^i)\in \mathcal{K}^i$ by 
$x^{i}(s^i,a^i)$ and add over all $(s^i,a^i)\in \mathcal{K}^i$ and by using the constraints 
$(ii)$-$(vi)$ and \eqref{IND_NASH_eq} we have
for each $i=1,2,\cdots,N$
\begin{align*}
\sum_{(s,a)\in\mathcal{K}}\left(\prod_{j=1}^N x^{j*}(s^j,a^j)\right)c^i(s,a)\leq \sum_{(s,a)\in\mathcal{K}}x^i(s^i,a^i)\left(\prod_{j=1;j\neq i}^N x^{j*}(s^j,a^j)\right)c^i(s,a)
\end{align*}
for all $i$-feasible $(x^i,x^{-i*})$ . In other words we can say that for each $i=1,2,\cdots,N$ 
\begin{equation*}\label{IND_NE}
 C_{ea}^i(f^*)\leq C_{ea}^i(f^i,f^{-i*}),\;\;\forall\; i\mbox{-feasible}\; (f^i,f^{-i*}).
\end{equation*}  
That is $(f^{i*})_{i=1}^N$ is Nash equilibrium of the constrained stochastic game 
$G_{ea}^c$ where 
\begin{gather*}
 f^{i*}(s^i,a^i)=\frac{x^{i*}(s^i,a^i)}{\sum_{a^i\in A^i(s^i)}x^{i*}(s^i,a^i)}
\end{gather*}
for all  $s^i\in S^i$, $a^i\in A^i(s^i)$, $i=1,2,\cdots,N$ whenever the denominator is non-zero 
(when it is zero $f^{i*}(s^i)$ is chosen arbitrarily from $\wp(A^i(s^i))$).

\end{proof}

\begin{remark}
 It is easy to see that [MP3] is also a non-convex constrained optimization problem.
\end{remark}

\subsubsection{Special cases}
We consider two special cases. First, we consider two player nonzero sum constrained stochastic game as defined in Section \ref{IND_model}, where, the constraints of both the players are decoupled. Next, we consider two player zero sum game as considered in \cite{Miller}.
\subsubsection*{(i) The case of two player game with decoupled constraints}
Here we consider the situation where there are only two players and the constraints of each player 
do not depend on the strategies of the other player.
This is possible when immediate costs of each player 
which correspond to his constraints do not depend 
on the state and actions of the other player, i.e., 
$d^{i,k}(s^1,s^2,a^1,a^2)=d^{i,k}(s^i,a^i)$  for all $s^i\in S^i$, $a^i\in A^i(s^i)$, $k=1,2,\cdots,n_i, i=1,2$.
We see that the mathematical program [MP3] reduces to a quadratic
program [QP3] as given below
{\allowdisplaybreaks
\begin{align*}
&\textup{{[QP3]}} \ \ \min\sum_{i=1}^2\left[\sum_{(s^1,a^1,s^2,a^2)}\left(\prod_{j=1}^2 x^j(s^j,a^j)\right)c^i(s^1,s^2,a^1,a^2)
-\left(v^i-\sum_{k=1}^{n_i}\delta_k^i\xi_k^i\right)\right]\\
&\text{s.t.}\\
&(i) ~ v^1+u^1(s^1)\leq \sum_{(s^2,a^2)\in \mathcal{K}^2} c^1(s^1,s^2,a^1,a^2)x^2(s^2,a^2)
+\sum_{k=1}^{n_1}d^{1,k}(s^1,a^1)\delta_k^1\\
&\hspace{3cm}+\sum_{\bar{s}^1\in S^1}p^1(\bar{s}^1|s^1,a^1)u^1(\bar{s}^1), \;\; \forall\; s^1\in S^1,\; a^1\in A^1(s^1)\\
& (ii) ~v^2 + u^2(s^2)\leq\sum_{(s^1,a^1)\in \mathcal{K}^1} c^2(s^1,s^2,a^1,a^2)x^1(s^1,a^1)
+\sum_{k=1}^{n_2}d^{2,k}(s^2,a^2)\delta_k^2\\
&\hspace{3cm}+\sum_{\bar{s}^2\in S^2}p^2(\bar{s}^2|s^2,a^2)u^2(\bar{s}^2), \;\; \forall\; s^2\in S^2,\; a^2\in A^2(s^2)\\
&(iii) ~ \sum_{(s^i,a^i)\in \mathcal{K}^i}\left(\delta(s^i,\bar{s}^i)- p^i(\bar{s}^i|s^i,a^i)\right)x^i(s^i,a^i)= 0, \;\; \forall\;  \bar{s}^i\in S^i, i=1,2 \\
& (iv) ~ \sum_{(s^i,a^i)\in \mathcal{K}^i}x^i(s^i,a^i)=1, \; \; \forall\; i=1,2\\
&(v) ~ \sum_{(s^i,a^i)\in \mathcal{K}^i}d^{i,k}(s^i,a^i)x^i(s^i,a^i)\leq \xi_k^i, \;\; \forall \; k=1,2,\cdots,n_i,\; i=1,2\\
&(ix) ~ x^i(s^i,a^i)\geq 0,\; \; \forall\; s^i\in S^i, a^i\in A^i(s^i),\; i=1,2\\
&(x) ~ \delta_k^i\geq 0, \;\; \forall \; k=1,2,\cdots,n_i,\; i=1,2.
\end{align*}
}
 

\subsubsection*{(ii) Zero sum constrained stochastic game \cite{Miller}}
As a further special case of constrained stochastic game $G_{ea}^c$ we consider two player zero sum game 
with decoupled constraints \cite{Miller}. This class of games with both unichain and multichain structure on the state processes of both 
the players can be solved by linear programs \cite{Miller}.  For zero sum case
simply set $c^1(s^1,s^2,a^1,a^2)=-c^2(s^1,s^2,a^1,a^2)=c(s^1,s^2,a^1,a^2)$ for all $s^1\in S^1$, $s^2\in S^2$, $a^1\in A^1(s^1)$, 
$a^2\in A^2(s^2)$ then the quadratic program
[QP3] can be separated into a primal-dual pair of linear programs which are same as given in \cite{Miller} in unichain case.

\subsection{A numerical example}\label{IND_num}
In this section we give one numerical example of a two player game where
constraints of both the players are decoupled. We compute the Nash equilibrium of this game  
by solving quadratic program [QP3]. 
The components of the stochastic game are: 
\begin{enumerate}
\item The state space of player 1 and player 2 are $S^1=\{1,2\}$, $S^2=\{3,4\}$ respectively.
\item The action sets of player 1  are $A^1(s^1)=\{1,2\}$ for all $s^1\in S^1$ and action sets of player 2 are 
      $A^2(s^2)=\{1,2\}$ for all $s^2\in S^2$.
\item The immediate costs of both the players, which are involved in their expected average costs they want to minimize, 
are given in Tables \ref{IND_CT1},  \ref{IND_CT2},  \ref{IND_CT3} and \ref{IND_CT4}.
\begin{table}[ht]%
\caption{Immediate costs}
\centering
\subtable [$(s^1,s^2)=(1,3)$]{\label{IND_CT1}
\begin{tabular}{|c|c|c|}\hline
\backslashbox{$a^1$}{$a^2$} & $a^2=1$ & $a^2=2$ \\ \hline
$a^1=1$ & (2,3) &(3,1) \\ \hline
$a^1=2$ & (4,2) & (2,4) \\ \hline
\end{tabular}}
\hspace{1cm}
\subtable [$(s^1,s^2)=(1,4)$]{\label{IND_CT2}
\begin{tabular}{|c|c|c|}\hline
\backslashbox{$a^1$}{$a^2$} & $a^2=1$ & $a^2=2$ \\ \hline
$a^1=1$ & (5,2) &(3,4) \\ \hline
$a^1=2$ & (3,2) & (4,1) \\ \hline
\end{tabular}}
\hspace{1cm}
\subtable [$(s^1,s^2)=(2,3)$]{\label{IND_CT3}
\begin{tabular}{|c|c|c|}\hline
\backslashbox{$a^1$}{$a^2$} & $a^2=1$ & $a^2=2$ \\ \hline
$a^1=1$ & (3,5) &(4,6) \\ \hline
$a^1=2$ & (5,2) & (2,1) \\ \hline
\end{tabular}}
\hspace{1cm}
\subtable [$(s^1,s^2)=(2,4)$]{\label{IND_CT4}
\begin{tabular}{|c|c|c|}\hline
\backslashbox{$a^1$}{$a^2$} & $a^2=1$ & $a^2=2$ \\ \hline
$a^1=1$ & (4,5) &(3,1) \\ \hline
$a^1=2$ & (1,2) & (4,3) \\ \hline
\end{tabular}}
\end{table}
These tables summarize the immediate costs of both the players in all the possible states.  For example in Table \ref{IND_CT1}
the entry $(2,3)$ represent $2$ as immediate cost of player 1 when first player is in state $1$ and he chooses action $1$ and second player is in state $3$ and  chooses action $1$. Similar explanation is for $3$ and other entries in all the tables.  

\item The transition probabilities of first and second Markov chains (one for each player) are given in the Tables \ref{IND_TP1}
and \ref{IND_TP2} respectively. We can easily check that both the Markov chains are unichain. In first Markov chain 
state $1$ is transient for some strategies of player 1 and state $2$ is recurrent for every strategy $f^1$ of player 1. In the second Markov chain both the states $3$ and $4$ are recurrent for every strategy $f^2$ of player 2.  So, the assumption (A1) is satisfied.
\begin{table}[ht]%
\caption{Transition probabilities of both the Markov chains}
\centering
\subtable[$p^1(.|s^1,a^1)$]{\label{IND_TP1}
\begin{tabular}{|c|c|c|}\hline
 {} & $a^1=1$ & $a^1=2$ \\ \hline
$s^1=1$ & (0.5,0.5) & (0.33,0.67)\\ \hline 
$s^1=2$ & (1,0) & (0,1) \\ \hline
\end{tabular}}
\hspace{1cm}
\subtable[$p^2(.|s^2,a^2)$]{\label{IND_TP2}
\begin{tabular}{|c|c|c|}\hline
 {} & $a^2=1$ & $a^2=2$ \\ \hline
$s^2=3$ & (0.67,0.33) & (0.4,0.6)\\ \hline 
$s^2=4$ & (0.25,0.75) & (1,0) \\ \hline
\end{tabular}}
\end{table}

\item Each player has one constraint. The immediate costs of both the players which are used in their
expected average constraints are given in Table \ref{IND_const1} and \ref{IND_const2}.
\begin{table}[ht]%
\caption{Immediate costs defining constraints}
\centering
\subtable[$d^1(s^1,a^1)$]{\label{IND_const1}
\begin{tabular}{|c|c|c|}\hline
{} & $a^1=1$ & $a^1=2$ \\ \hline
$s^1=1$ & 7 & 4\\ \hline
$s^1=2$ & 2 & 5\\ \hline 
\end{tabular}}
\hspace{1cm}
\subtable[$d^2(s^2,a^2)$]{\label{IND_const2}
\begin{tabular}{|c|c|c|}\hline
{} & $a^2=1$ & $a^2=2$ \\ \hline
$s^2=3$ & 4 & 3\\ \hline
$s^2=4$ & 3 & 5\\ \hline 
\end{tabular}}
\end{table} 
\item The bounds defining the constraints are $\xi^1=5$, $\xi^2=3.5$.
\end{enumerate}
We solve the quadratic program [QP3], corresponding to the above data, by using MATLAB and obtain   
\begin{gather*}
 \zeta^*=(1.2941,0,0,1.7059,-0.5882,0.5882,0,0,0,1,0,0.2941,0.7059,0,0,0).
\end{gather*}
Note that at $\zeta^*$ the objective function value is zero and hence it is the global minimum of  
of [QP3]. We have $x^{1*}=(0,0,0,1)$ and $x^{2*}=(0,0.2941,0.7059,0)$ then from Theorem \ref{IND_Main_thm} $(b)$ the Nash equilibrium 
$(f^{1*},f^{2*})$ of constrained stochastic game $G_{ea}^c$, where
\[
f^{1*}=((\alpha,1-\alpha), (0,1)) \; \; \mbox{for all} \;  \alpha\in [0,1] 
\]
\[
 f^{2*}=((0,1),(1,0))
\]
Note that under $f^{1*}$ player 1 can use any randomized strategy at state 1 which comes from the fact that state 1 is transient under $f^{1*}$.
The costs of both the players at Nash equilibrium $(f^{1*},f^{2*})$ are 
\begin{gather*}
C_{ea}^1(f^{1*},f^{2*})= 1.2941\\
C_{ea}^2(f^{1*},f^{2*})=1.7059.
\end{gather*}

\bibliographystyle{elsart-num}
\bibliography{BibliographyExample}

\appendix
\section*{Appendix A}

\textbf{A single mathematical program for average and discounted cost criteria model}

The mathematical programs [MP1] and [MP2] that characterize the stationary Nash equilibria of single controller constrained stochastic games with  average and discounted cost criteria respectively can be recovered from one mathematical program [MP4] given below.
{\allowdisplaybreaks
\begin{align*}
& \textup{{[MP4]}}\quad \min_{\eta}\left[\left(f^T\boldsymbol{C}^1x-\left(\textbf{1}^T z -(\delta^1)^T\xi^1\right)\right)
+\left(f^T\boldsymbol{C}^2x-\left(v+(1-\beta)\gamma^T u-(\delta^2)^T\xi^2\right)\right)\right] \\
\text{s.t.}\\
& (i) ~ v+ u(s)\leq \left[(f(s))^T\boldsymbol{C}^2(s)\right]_{a^2}
+\sum_{l=1}^{n_2}\delta^2_l \left[(f(s))^T\boldsymbol{D}^{2,l}(s)\right]_{a^2}\\
&\hspace{3cm}+\beta\sum_{s'\in S} p(s'|s,a^2)u(s'), \; \; \forall \;\; s\in S,\; a^2\in A^2(s)\\
& (ii) ~ z(s)\leq \left[\boldsymbol{C}^1(s)x(s)\right]_{a^1}+\sum_{k=1}^{n_1}\delta_k^1 d^{1,k}_{sub}(s,a^1),\; \; \forall \;\; s\in S,\; a^1\in A^1(s)\\
& (iii)~\sum_{(s,a^2)\in \kappa^2}\left[\delta(s,s')-\beta p(s'|s,a^2)\right]x(s,a^2)=(1-\beta)\gamma(s'), \;\;\forall \;\; s'\in S\\
& (iv) ~ \sum_{(s,a^2)\in \kappa^2}x(s,a^2)=1\\
& (v) ~ \sum_{(s,a^1)\in \kappa^1}d^{1,k}_{sub}(s,a^1)f(s,a^1)\leq \xi_k^1, \;\;\forall \;\; k=1,2,\cdots,n_1\\
& (vi)  ~ \sum_{s\in S} (f(s))^T\boldsymbol{D}^{2,l}(s)x(s)\leq \xi_l^2, \;\;\forall \;\; l=1,2,\cdots,n_2\\
& (vii) ~ \sum_{a^1\in A^1(s)} f(s,a^1)=1,  \;\;\forall \;\; s\in S\\
& (viii) ~ f(s,a^1)\geq 0, \;\;\forall \;\; s\in S,\;a^1\in A^1(s) \\ 
& (ix) ~   x(s,a^2)\geq 0, \;\; \forall \;\; s\in S,\;a^2\in A^2(s)\\
& (x)  ~ \delta_k^1\geq 0, \;\; \forall \;\; k=1,2,\cdots,n_1\\
& (xi) ~ \delta^2_l\geq 0,  \;\; \forall \;\; l=1,2,\cdots,n_2.
\end{align*}
}
The mathematical program [MP1] can be obtained by putting $\beta=1$ in [MP4].
For discount factor $\beta\in [0,1)$ the constraint $(iv)$ of [MP4] is redundant because it can be obtained by summing $(iii)$ over all $s'\in S$ and hence the variable $v$ is also redundant. So, by removing constraint $(iv)$ and variable $v$ from [MP4] we obtain [MP2].

\end{document}